\documentclass[ a4paper]{amsart}
\usepackage{amsmath, amssymb,amsthm}

\usepackage{etoolbox}
\makeatletter
\patchcmd{\appendix}{\@Alph}{\@Roman}{}{}
\makeatother

\usepackage{color}
\def\comment#1{}

\usepackage{hyperref}
\usepackage{txfonts}
\usepackage{enumerate}
\usepackage{caption}
\usepackage{accents}

\usepackage{graphicx}
\usepackage{float}
\graphicspath{ {img/} }

\newtheorem{theorem}{Theorem}

\newtheorem{definition}[theorem]{Definition}
\newtheorem{lemma}[theorem]{Lemma}
\newtheorem{proposition}[theorem]{Proposition}

{Important Convention}

\theoremstyle{remark}
\newtheorem{remark}[theorem]{Remark}
\newtheorem{example}[theorem]{Example}

 \newcommand{\m}{\mathbf{m}}

 \renewcommand{\phi}{\varphi}

\newcommand{\E}{\mathbb{E}}
\renewcommand{\P}{\mathbb{P}}
\newcommand{\N}{\mathbb{N}}

\newcommand{\R}{\mathbb{R}}

\DeclareMathOperator*{\argmin}{argmin}
\newcommand{\bes}{\begin{subequations}}
\newcommand{\ees}{\end{subequations}}
\newcommand{\eea}{\end{eqnarray}}

\newcommand{\NN}{{\mathbb N}}

\newcommand{\RR}{{\mathbb R}}

\newcommand{\cal}{\mathcal}

\usepackage{bbm}

\renewcommand{\epsilon}{\varepsilon}

\newcommand{\fourIdx}[5]{%
\setbox1=\hbox{\ensuremath{^{#1}}}%
 \setbox2=\hbox{\ensuremath{_{#2}}}%
 \setbox5=\hbox{\ensuremath{#5}}%
 \hspace{\ifnum\wd1>\wd2\wd1\else\wd2\fi}%
 \ensuremath{\copy5^{\hspace{-\wd1}\hspace{-\wd5}#1\hspace{\wd5}#3}%
 _{\hspace{-\wd2}\hspace{-\wd5}#2\hspace{\wd5}#4}%
 }}

\numberwithin{equation}{section}
\numberwithin{theorem}{section}

\newcommand{\Opt}{\mathsf{Opt}}

\newcommand{\opt}{G}

\renewcommand{\subset}{\subseteq}

\usepackage{subcaption}
\usepackage{graphicx}

\begin{document}

\title{Stochastic Gradient Descent for  Barycenters in Wasserstein space}

 \author{Julio Backhoff-Veraguas}
\address{Faculty of Mathematics, University of Vienna}
\email{julio.backhoff@univie.ac.at}

\author{Joaquin Fontbona}
\address{Center for Mathematical Modeling, and Department of Mathematical Engineering, Universidad de Chile}
\email{fontbona@dim.uchile.cl}

\author{Gonzalo Rios}
\address{NoiseGrasp SpA}
\email{grios@noisegrasp.com}

\author{Felipe Tobar}
\address{Initiative for Data \& Artificial Intelligence, and Center for Mathematical Modeling, Universidad de Chile}
\email{ftobar@uchile.cl}
\urladdr{http://www.dim.uchile.cl/~ftobar/}

  \begin{abstract}
 
We present and study a novel algorithm for the computation of 2-Wasserstein population barycenters { of absolutely continuous probability measures on Euclidean space}. The proposed method can be seen as a stochastic gradient descent procedure in the 2-Wasserstein space,  as well as a manifestation of a Law of Large Numbers therein. The algorithm aims at finding a Karcher mean or critical point in this setting, and can be implemented ``online", sequentially using i.i.d.  random measures sampled from the population law.  We provide natural sufficient conditions for this algorithm to a.s.\  converge in the Wasserstein space towards the population barycenter,  { and we introduce a novel, general condition which ensures uniqueness of Karcher means and moreover allows us to obtain explicit, parametric convergence rates  for the expected optimality gap. We furthermore study the mini-batch version of this algorithm, and discuss examples of families of population laws to which our method and results can be applied. This work expands and deepens ideas and results introduced in an early version of \cite{backhoff2018bayesian}, in which a statistical application (and numerical implementation) of this method is developed in the context of Bayesian learning.}

\medskip



\medskip

\noindent \today

\end{abstract}

\maketitle

\section{Introduction}
\label{sec:SGDW}

Let $\mathcal P(\mathcal X)$ denote the space of Borel probability measures over a Polish space $\mathcal{X}$. Given  $\mu,\upsilon\in \mathcal P(\mathcal X)$,  denote  by 
$$\Gamma ( \mu ,\upsilon ) :=\{\gamma \in \mathcal{P}(\mathcal{X} \times \mathcal{X}): \,\gamma ( dx,\mathcal{X}) =\mu(dx) ,\gamma ( \mathcal{X},dy) =\upsilon(dy) \},$$ the set of couplings (transport plans) with marginals $\mu$ and $\upsilon$. For a fixed, compatible complete metric $d$, and given a real number $p\geq 1$, we define the $p$-Wasserstein space by 
\begin{align*}\textstyle
 \mathcal W_p(\mathcal{X})& \textstyle :=\left\{\eta\in \mathcal P(\mathcal{X}): \int_{\mathcal{X}}d(x_0,x)^p \eta(dx)< \infty,\,\,\text{some }x_0  \right\}.
\end{align*}
Accordingly, the $p$-Wasserstein distance between measures $\mu,\upsilon\in \mathcal W_p(\mathcal{X})$ is given by
\begin{align} \textstyle
	 W_{p}( \mu ,\upsilon) = \left(\inf_{\gamma \in \Gamma ( \mu	,\upsilon ) }\int_{\mathcal{X\times X}}d(x,y)^{p}\gamma (dx,dy) \right) ^{\frac{1}{p}}.\label{eq W p}
\end{align}
When $p=2$, $\mathcal{X}=\R^q$, $d$ is the Euclidean distance  and $\mu$ is absolutely continuous, {which is the setting which we will soon adopt for the remainder of the paper}, Brenier's theorem \cite[Theorem 2.12(ii)]{villani2003topics} establishes the uniqueness of a minimizer for the r.h.s.~of eq.~\eqref{eq W p}. Furthermore, this optimizer is supported on the graph of the gradient of a convex function. See \cite{ags08} and \cite{villani2008optimal} for further general background on optimal transport.   

We recall now the definition of the Wasserstein population barycenter:

\begin{definition}\label{defi bary pop}
	Given $\Pi \in \mathcal{P}(\mathcal{P}(\mathcal{X}))$, denote 
	\begin{align*}
	V_p(\bar{m}) &:=\textstyle  \int_{\mathcal{P}(\mathcal X)} W_{p}(m,\bar{m})^{p}\Pi(dm).
	\end{align*}
Any measure  $\hat{m}\in \mathcal W_p(\mathcal{X})$ which is a minimizer of the problem $$\textstyle \inf_{\bar{m} \in \mathcal{P}(\mathcal X)} V_p(\bar{m}),$$
is called a $p$-Wasserstein population barycenter of $\Pi$.  
\end{definition}

\textcolor{black}{Wasserstein barycenters were first introduced and analyzed by Carlier and Agueh in \cite{agueh2011barycenters}, in the case when the support of $\Pi\in \mathcal{P}(\mathcal{P}(\mathcal{X}))$ is finite and $\mathcal X=\R^q$. More generally, Bigot and Klein \cite{bigot2018characterization}  and  Le Gouic and Loubes \cite{le2017existence} considered the so-called \emph{population barycenter}, i.e., the general case in Definition \ref{defi bary pop} where $\Pi$ may have infinite support. These works addressed among others, the basic questions of existence and uniqueness of solutions. See also Kim and Pass \cite{kim2017wasserstein} for the Riemannian case. The concept of Wasserstein barycenter has been extensively studied from both theoretical and practical perspectives during the last decade: we refer the readers to Panaretos and Zemel's overview \cite{PaZe20} for statistical applications and   to the works of   {\color{black}Cuturi, Peyr\'e and coauthors \cite{CuPe19,CuPe16,cuturi2014fast,NIPS2021}} for computational aspects of optimal transport and applications in machine learning.}

In this article we develop a stochastic gradient descent (SGD)  algorithm for the computation of $2$-Wasserstein population barycenters. The method inherits features of the SGD rationale, in particular:

\begin{itemize}
\item It exhibits a reduced computational cost compared to methods based on the direct formulation of the barycenter problem using all the available data, as SGD only considers a limited number of samples of $\Pi\in \mathcal{P}(\mathcal{P}(\mathcal{X}))$ per iteration.
\item It is \emph{online} (or \emph{continual}, as referred to in the machine learning community), meaning that it can incorporate additional datapoints sequentially to update the barycenter estimate, whenever new observations becomes available. This is relevant even when $\Pi$ is finitely supported.
\item   { Conditions ensuring convergence towards the Wasserstein barycenter as well as finite dimensional convergence rates for the algorithm can be provided. Moreover, the variance of the gradient estimators   can be reduced by using mini-batches.}
\end{itemize}

\smallskip 

From now on we assume that
\begin{enumerate}
\item[(A1)]  $\mathcal X=\R^q$, $d$ is the squared Euclidean metric  and $p=2$. 
\end{enumerate}
\textcolor{black}{We denote by $\mathcal W_{2,ac}(\mathcal X)$ the subspace of $\mathcal W_{2}(\mathcal X)$ of absolutely continuous measures with finite second moment, and for any $\mu\in\mathcal W_{2,ac}(\mathcal X) $ and $\nu\in\mathcal W_{2}(\mathcal X) $ we write $T_\mu^\nu$ for the ($\mu$-a.s.\ unique) gradient of a convex function such that\footnote{Notice that $x\mapsto T_{\mu}^{\nu}(x)$ is a $\mu$-a.s.\ defined map and that we use throughout the notation $T_\mu^\nu(\rho)$ for the image measure/law of this map under the measure $\rho$.} $T_\mu^\nu(\mu)=\nu$.} 

\smallskip

Recall that a set $B \subset\mathcal W_{2,ac}(\mathcal X)$  is \emph{geodesically convex}, if   for every $\mu,\nu\in B$ and $t\in[0,1]$ we have $((1-t)I+tT_\mu^\nu)(\mu)\in B$, with $I$ denoting the identity operator.  We will also assume the following condition for most of the article: 
\begin{enumerate}
\item[(A2)] $\Pi$ gives full measure to a geodesically convex $W_2$-compact set $K_\Pi\subset\mathcal W_{2,ac}(\mathcal X)$. 
\end{enumerate}
In particular, under (A2) we have $\int_{\mathcal P(\mathcal X)} \int_{\mathcal X} d(x,x_0)^2m(dx)\Pi(dm)<\infty$ for all $x_0$.  Moreover,  for each $\nu \in \mathcal W_2(\mathcal X)$ and $\Pi(dm) $ a.e.\ $m$, there is a unique optimal transport map $T_m^\nu$ from $m $ to  $\nu$ and, {\color{black} by  \cite[Proposition 6]{le2017existence}, the  2-Wasserstein population barycenter is unique}. 

\begin{definition}
	Let {\color{black}$\mu_0 \in K_\Pi$}, $m_k \stackrel{\text{iid}}{\sim}  \Pi$,  and $\gamma_k > 0$ for $k \geq 0$. We define the stochastic gradient descent (SGD) sequence by
	\begin{align}
	\label{eq:sgd-seq}
	\mu_{k+1} := \left[(1-\gamma_k)I + \gamma_k T_{\mu_k}^{m_k}\right](\mu_k) \text{ , for } k \geq 0.
	\end{align}
\end{definition}
The reasons as to why we can truthfully refer to the above sequence as stochastic gradient descent will become apparent in Sections \ref{sec grad desc} and \ref{sec stoch grad}.  We stress that the sequence is a.s.\ well-defined, as one can show by induction that $\mu_k\in \mathcal W_{2,ac}(\mathcal X)$ a.s.~ {\color{black} thanks to Assumption (A2). We also refer to Section  \ref{sec stoch grad}  for remarks on the measurability of the random maps $\{T_{\mu_k}^{m_k}\}_k$ and sequence $\{\mu_k\}_k$.}

 Throughout the article, we assume the following conditions on the steps $\gamma_k$ in eq.~\eqref{eq:sgd-seq}, commonly required for the convergence of SGD methods :
\begin{align}
\label{eq:cond1-gamma}
 \textstyle \sum_{k=1}^{\infty}\gamma_k^2 &< \infty\\
\label{eq:cond2-gamma}
\textstyle \sum_{k=1}^{\infty}\gamma_k &= \infty.
\end{align}

In addition to barycenters, we will need the concept of \emph{Karcher means} (c.f.\ \cite{panaretos2017frechet}), which, in the setting of the optimization problem in $\mathcal W_{2,ac}(\mathcal X)$ considered here,  can be intuitively  understood as  an analogue of a critical  point {of a smooth function in Euclidean space  (see the discussion in Section  \ref{sec grad desc}).}

\begin{definition}\label{defi:Karcher_intro}
Given  $\Pi \in \mathcal{P}(\mathcal{P}(\mathcal{X}))$, we say that $\mu\in\mathcal W_{2,ac}(\mathcal X)$ is a Karcher mean of $\Pi$ if
$$\textstyle\mu\left( \left\{x: x =\int_{m\in\mathcal P(\mathcal X)} T_\mu^m(x)\Pi(dm)\right\} \right)=1.$$ 
\end{definition}

It is known that any 2-Wasserstein barycenter is a Karcher mean, though the latter is in general a strictly larger class, see \cite{alvarez2016fixed} or Example \ref{ex:non_unique_Karcher} below. However, if there is a unique Karcher mean, then it must coincide with the unique barycenter.
We can now state the main result of the article. 
\textcolor{black}{
\begin{theorem}
	\label{thm:sgd-convergence}
	Assume (A1), (A2), conditions \eqref{eq:cond1-gamma} and \eqref{eq:cond2-gamma}, and  that the $2-$Wasserstein barycenter $\hat{\mu}$ of $\Pi$ is the unique Karcher mean. Then, the SGD sequence $\{\mu_k\}_k$ in eq.~\eqref{eq:sgd-seq} is a.s.~$\mathcal W_2$-convergent to $\hat{\mu}\in K_\Pi$.
\end{theorem}}

An interesting aspect of Theorem \ref{thm:sgd-convergence} is that it hints at a Law of Large Numbers (LLN) on the 2-Wasserstein space. Indeed,  in the conventional LLN, for iid samples $X_i$ the summation $S_{k}:=\frac{1}{k}\sum_{i\leq k}X_i$ can be expressed as $$ \textstyle S_{k+1}=\frac{1}{k+1}X_{k+1}+\left( 1-\frac{1}{k+1} \right)S_k.$$ Therefore, if we rather think of sample $X_k$ as a measure $m_k$ and of $S_k$ as $\mu_k$, we immediately see the connection with $$\textstyle\mu_{k+1} := \left[ \frac{1}{k+1} T_{\mu_k}^{m_k}+\left(1- \frac{1}{k+1} \right)I \right](\mu_k),$$
obtained from eq.~\eqref{eq:sgd-seq} when we take $\gamma_k=\frac{1}{k+1}$. The convergence in Theorem \ref{thm:sgd-convergence} can thus be interpreted as an analogy to the convergence of $S_k$ to the mean of $X_1$ (we thank Stefan Schrott for this observation).

 \smallskip 

 In order to state our second main result, Theorem \ref{thm:uKarcherSGDrate}, we introduce a new concept:

\begin{definition}\label{def:pseudo-asoc}
Given  $\Pi \in \mathcal{P}(\mathcal{P}(\mathcal{X}))$ we say that  a Karcher mean  $\mu\in\mathcal W_{2,ac}(\mathcal X)$  of $\Pi$  is pseudo-associative if there exists $C_{\mu}>0$ such that for all $\nu \in \mathcal W_{2,ac}({\cal X})$, one has 
\begin{equation}\label{eq:pseudo-asoc_ineq} \textstyle
W_2^2(\mu, \nu)\leq C_{\mu}  \int_{\mathcal X}   \left |     \int_{\mathcal P(\mathcal X)} (T_\nu^{m}(x)-x) )\Pi(dm) \right |^2  \nu(dx) .
\end{equation}
\end{definition}

Since term on the r.h.s.\ of \eqref{eq:pseudo-asoc_ineq} vanishes for any Karcher mean $\nu$, existence of a pseudo-associative Karcher mean implies $\mu=\nu$, hence uniqueness of Karcher means.   We will see that moreover the existence of a pseudo-associative Karcher mean implies a Polyak-Lojasiewicz inequality for the functional minimized by the barycenter, see \eqref{eq:PL}. This in turn can be  utilized  to obtain convergence rates for the expected optimality gap in a similar way as in the Euclidean case. While the pseudo-associativity condition is a strong requirement, in the following result we are able to weaken Assumption (A2) into the milder:
\begin{enumerate}
\item[(A2')] $\Pi$ gives full measure to a geodesically convex set $K_\Pi\subset\mathcal W_{2,ac}(\mathcal X)$ which is $\mathcal W_2$-bounded (i.e.\ $\sup_{m\in K_\Pi}\int |x|^2m(dx)<\infty$). 
\end{enumerate}
Clearly this condition is equivalent to requiring that the support of $\Pi$ be  $\mathcal W_2$-bounded and consist of absolutely continuous measures. 
Our second main result is:

\begin{theorem}\label{thm:uKarcherSGDrate}
Assume (A1), (A2') and that the $2-$Wasserstein barycenter $\hat{\mu}$ of $\Pi$ is a pseudo-associative Karcher mean. Then, $\hat{\mu}$ is the unique barycenter of $\Pi$,  and the SGD sequence $\{\mu_k\}_k$ in eq.~\eqref{eq:sgd-seq} is a.s.~weakly convergent to $\hat{\mu}\in K_\Pi$ as soon as \eqref{eq:cond1-gamma} and \eqref{eq:cond2-gamma}  hold. Moreover, for every $a>C_{\hat{\mu}}^{-1}$ and $b\geq a$ there exists an explicit constant  $C_{a,b}>0$ such that, if $\gamma_k=\frac{a}{b+k}$ for all $k\in \NN$, the expected optimality gap satisfies: 
$$ \textstyle \mathbb{E}\left[F(\mu_{k}) - F(\hat{\mu}) \right] \leq \frac{C_{a,b}}{b+k} \, .$$
\end{theorem}

 Let us explain the reason for  the terminology ``pseudo-associative'' we have chosen. This comes from the fact that the inequality in  Definition \ref{def:pseudo-asoc} holds true (with equality and $C_{\mu}=1$) {as soon as the following associativity property holds: 
 $$T_{\mu}^{m}(x) = T_{\nu}^m \circ T_{\mu}^{\nu} (x) $$
$\mu(dx)$ a.s., for each pair $\nu,m\in W_{2,ac}(\mathcal X)$, see Remark \ref{rem:assocKarcher} below. 
The previous identity was assumed to  hold true for the results proved in \cite{bigot2018characterization}, and is always valid in $\R$, since the composition of monotone functions is monotone. Further examples where the associativity property holds  are discussed in  Section \ref{sec special cases}.  Thus, in all those settings,  \eqref{eq:pseudo-asoc_ineq}  and  Theorem \ref{thm:uKarcherSGDrate} hold   as soon as assumption (A2')  is granted, and we further have the explicit LLN-like expression $
\mu_{k+1}= \textstyle \left(\frac{1}{k}\sum_{i=1}^k T_{\mu_0}^{m_i}\right)(\mu_0)$. As regards the more general pseudo-associativity property,  we will see that it holds  for instance in the Gaussian framework studied in \cite{chewi2020gradient}, and in certain classes of scatter-location families, which we discuss in Section \ref{sec special cases}.}

 \smallskip 
 
The paper is organized as follows: 

\begin{itemize}
\item In Section \ref{sec grad desc} we recall basic ideas and results on gradient descent algorithms for Wasserstein barycenters. 
\item Section \ref{sec stoch grad} we prove  Theorem \ref{thm:sgd-convergence}, previously providing technical elements required to that end.

\item {In Section \ref{sec unique Karcher SGD rates} we discuss the notion of pseudo-associative Karcher means and its relations to the so-called Polyak- Lojasiewicz  and  variance inequalities, and we prove Theorem
\ref{thm:uKarcherSGDrate}. }

\item In Section \ref{sec:generalization-sgd} we introduce the mini-batch version of our algorithm, discuss how  this improves the variance of gradient-type estimators, and state  extensions of our previous results to that setting.  

\item In Section \ref{sec special cases} we consider closed-form examples and explain how in these cases the existence of pseudo-associative Karcher means required in  Theorem
\ref{thm:uKarcherSGDrate} can be guaranteed. W also explore certain properties of probability distributions which are ``stable'' under the operation of taking their barycenters.
\end{itemize}

\medskip 

{ In the remainder of this introduction, we  provide  independent discussions on various aspects of our results  (some of them suggested by the referees) and on related literature.}

\smallskip 


{{\bf On the assumptions.}  Although   Assumption (A2) might appear as strong at first sight, it can be guaranteed in suitable parametric situations (e.g., Gaussian, or the scatter-location  setting recalled in Section \ref{sec scat loc}) or, more generally, under moment and density constraints on the measures in $K_\Pi$.  For instance if $\Pi$ is  supported on a finite set of  measures with finite Boltzmann entropy, then Assumption (A2) is guaranteed. More generally, if the support of $\Pi$ is 2-Wasserstein compact set  and the Boltzmann  entropy is uniformly bounded on it, then Assumption (A2) is fulfilled too: see Lemma \ref{lem:verification_A2} below. }

Conditions (A2) and of uniqueness of a  Karcher mean in  Theorem  \ref{thm:sgd-convergence} are natural  substitutes of,  respectively, the  compactness of SGD sequences and the uniqueness of critical points, which hold   true  under usual sets of assumptions ensuring the convergence of SGD  in Euclidean space 
 to a minimizer,  e.g. some growth control and certain  strict convexity-type condition at the  minimum. The usual reasonings underlying the convergence analysis of stochastic gradient descent in Euclidean spaces seem however not applicable in our context, as
 the functional $\mathcal W_{2,ac}(\mathcal X)\ni\mu\mapsto F(\mu):=\int W_{2}(m,\bar{m})^{2}\Pi(dm) $ is not convex, in fact not even $\alpha$-convex for some $\alpha\in\mathbb R$, when  $\mathcal W_{2,ac}(\mathcal X)$ is endowed with its almost Riemannian structure induced by optimal transport (see \cite[Ch.7.2]{ags08}). The function $F$ is neither $\alpha$-convex for some $\alpha\in\mathbb R$, when we use generalized geodesics (see \cite[Ch.9.2]{ags08}).  In fact,  SGD in finite-dimensional Riemannian manifolds could provide a more suitable framework to draw inspiration from, see e.g.  \cite{Bo13}.  Ideas useful in that setting seem not straightforward to leverage since that work either assumes negative curvature (while $\mathcal W_{2,ac}(\mathcal X)$ is positively curved), or that the functional to be minimized  be rather smooth and have bounded derivatives of first and second order.

{ The assumption that $K_\Pi$ is contained in a subset of $\mathcal W_{2}(\mathcal X)$ of absolutely continuous probability measures, and hence the existence of optimal transport maps between elements of $K_\Pi$, appears as a more structural requirement, as it is needed to construct the iterations \eqref{eq:sgd-seq}. Indeed,   by dealing with an extended notion of Karcher mean, it is  in principle  also possible to define an analogous iterative scheme in a general setting including in particular the case of discrete laws, and thus to relax to some extent the absolute continuity requirement. However, this introduces additional technicalities, and we unfortunately were not able to provide  conditions ensuring the convergence of the method in reasonably general situations. For completeness of the discussion, 
we sketch  the main ideas of this possible extension in the Appendix.}

 {\bf Uniqueness of Karcher means.} Regarding situations where uniqueness of Karcher means can be granted, we refer to  \cite{panaretos2017frechet,PaZe20}  for sufficient conditions when the support of $\Pi$ is finite, based on the regularity theory of optimal transport, and to  \cite{bigot2018characterization} for the  case of an infinite support, under rather strong assumptions. We  remark also that in one dimension, the uniqueness of Karcher means is known to hold  {without further assumptions. { A first counterexample where this uniqueness is not guaranteed is given in \cite{alvarez2016fixed}, see also the simplified Example \ref{ex:non_unique_Karcher} below.  A general understanding of the uniqueness of Karcher means remains  however an open and interesting challenge. This is not only relevant for the present work, but also for the (non-stochastic) gradient descent method of Zemel and Panaretos \cite{panaretos2017frechet} and the fixed-point iterations of \'Alvarez-Esteban, del Barrio, Cuesta-Albertos and Matr\'an \cite{alvarez2016fixed}. The notion of pseudo-associativity introduced in Definition \ref{def:pseudo-asoc} provides an alternative  viewpoint on this question, complementary to the aforementioned ones, which might deserve being further explored. }

{ {\bf Gradient descents in Wasserstein space.} Gradient descent (GD) in  Wasserstein space was introduced as a method to compute barycenters  in the mentioned works  \cite{alvarez2016fixed} and \cite{panaretos2017frechet}.   The SGD method we develop here was introduced  in an early version of the work \cite{backhoff2018bayesian} (available on  arXiv:1805.10833),} as a way to compute Bayesian  estimators based on 2-Wasserstein barycenters \footnote{ { In view of the independent, theoretical interest of this SGD method, we decided to separately present this algorithm here,  along with a deeper analysis and more complete results on it,    and devote  \cite{backhoff2018bayesian} exclusively to its statistical application and implementation.}}.}
{ More recently,   Chewi et al. \cite{chewi2020gradient} obtained convergence rates for the expected optimality gap  for these GD and SGD methods in the case of Gaussian families of distributions with uniformly bounded, uniformly positive definite covariance matrices. To that end they relied  on proving a Polyak- Lojasiewicz inequality, and also derived quantitative  convergence bounds in $W_2$ for the SDG sequence, relying on a variance inequality, which they showed to hold under general though strong conditions on the dual potential of the barycenter problem (verified under their assumptions).  We refer to the recent work \cite{CaDeMe22} for more on the variance  inequality.   }

The Riemannian-like structure of the Wasserstein space and its associated gradient have been utilized in various other ways with statistical or machine learning motivations in recent years; see e.g.\ \cite{ChBa18} for particle-like  approximations, \cite{KeLiBlGl21} for a sequential first order method, and \cite{LiMo18,ChLi20} for information geometry perspectives.

 {

 {\bf Computational aspects.}  Implementing our SGD algorithm   requires, in general, computing or approximating optimal transport maps between two given absolutely-continuous distributions (some exceptions where explicit closed form maps are available are given in Section  \ref{sec special cases}). During the last two decades, 
 considerable progress has been made on the numerical resolution of the latter problem  through PDE  methods \cite{loeper2005numerical,benamou2000computational,angenent2003minimizing,Bonneel2011} and, more recently, through entropic regularization approaches  \cite{cuturi2013sinkhorn,solomon2015convolutional,CuPe16, CuPe19,MaGeMi21,dognin2018wasserstein} crucially relying on the Sinkhorn algorithm \cite{Sinkhorn,SinkhornKnopp}, which significantly speeds up the approximate resolution of the problem. It is also possible to approximate  optimal transport maps via estimators built from samples (based on the Sinkhorn algorithm, plug-in estimators or using stochastic approximations) \cite{HuRi19,bercu2021asymptotic, PoNL21,DePrSe21,MaBaNLWa21}. We also refer to \cite{KoLiGeSoFiBu34} for an overview and comparison of sample-free methods for continuous distributions (for instance based on neural networks).}

 {

 {\bf  Applications.} The proposed method to compute Wasserstein barycenters is well suited to situations where a population law $\Pi$ on infinitely many probability measures is considered.   The Bayesian setting addressed in \cite{backhoff2018bayesian} provides a good example of such situation, i.e.\ where one needs to compute the Wasserstein barycenter of a prior / posterior distribution $\Pi$ on models which has a possibly infinite support. 
 
 Further instances of population laws $\Pi$ with infinite support arise in the context of Bayesian deep learning \cite{Wi20}, in which a neural network's (NN) weights are sampled from a given law (e.g., Gaussian, uniform, or Laplace); in the case these NNs parametrise probability distributions, the collection of resulting probability laws is distributed according to a law $\Pi$ with possibly infinite support. Another example is Variational AutoEncoders \cite{KiWe13}, in which case one models the parameters of a law by a simple random variable (usually Gaussian) which is then passed to a \textit{decoder} NN. In both cases, the support of $\Pi$ is infinite 
naturally. Furthermore, sampling from $\Pi$ is straightforward in these cases, which eases the implementation of the SGD algorithm. }
 
{\bf Possible extensions.} It is in principle possible to define and study stochastic gradient descent methods similar to \eqref{eq:sgd-seq} for the minimization on the space of measures  of other functionals than the barycentric objective function, or with respect to other geometries than the 2-Wasserstein one. For instance, the recent paper \cite{NIPS2021}  introduces the notion of barycenters of probability measures based on {\it weak optimal transport} \cite{gozlan2017kantorovich, backhoff2019existence, gozlan2020mixture} and  extends the ideas and algorithm developed  here to that setting. A further,  natural example of functionals to consider are the  entropy-regularized barycenters, dealt with for numerical purposes in some of the  aforementioned works and most recently in the articles by Chizat et al.\ \cite{Ch23,ChVa23}.

On a different vein, it would be interesting to study conditions ensuring the convergence of the algorithm \eqref{eq:sgd-seq} to stationary points (Karcher means), when the latter is a class that strictly contains the minimum (barycenter). Under suitable conditions, such convergence can be expected to hold by analogy with the behavior of the Euclidean SGD algorithm in  general  (non necessarily convex) settings \cite{bottou2018optimization}. In fact,  we believe this question can also be linked to the pseudo-associativity of Karcher means. A deeper study is left for future work.

\section{Gradient descent in Wasserstein space: a review}
\label{sec grad desc}

We first survey the gradient descent method for the computation of 2-Wasserstein barycenters. This method will serve as a  motivation for the subsequent development of the SGD in Section \ref{sec stoch grad}. For simplicity, we take $\Pi$ finitely supported. Concretely, we suppose in this section that $$\Pi= \textstyle \sum_{i\leq L}\lambda_i\delta_{m_i},\mbox{ with } L\in\N, \, \lambda_i\geq 0 \mbox{ and } \textstyle \sum_{i\leq L}\lambda_i=1.$$
We define the operator over absolutely continuous measures
\begin{align}
\label{eq:G-operator}
  G(m) \textstyle := \left(\sum_{i=1}^{L}\lambda_iT_m^{m_i}\right)(m).
\end{align}
Notice that the fixed-points of $G$ are precisely the Karcher means of $\Pi$ presented in the introduction. Thanks  to \cite{alvarez2016fixed} the operator $G$ is continuous for the ${W}_2$ distance. Also, if at least one of the $L$ measures $m_i$ has a bounded density, then the unique Wasserstein barycenter $\hat{m}$ of $\Pi$ has a bounded density as well and satisfies $G(\hat{m}) = \hat{m}$. This suggests to define, starting from $\mu_0$,  the sequence
\begin{align}
\label{eq:fixed-point}
  \mu_{n+1} := G(\mu_n), \text{ for } n \geq 0.
\end{align}
The next result was proven by \'Alvarez-Esteban, Barrio, Cuesta-Albertos, Matr\'an in \cite[Theorem 3.6]{alvarez2016fixed} and independently by Zemel and Panaretos in \cite[Theorem 3, Corollary 2]{panaretos2017frechet}:

\begin{proposition}
	The sequence $\{\mu_n\}_{n \geq 0}$ in eq.~\eqref{eq:fixed-point} is tight and every weakly convergent subsequence of $\{\mu_n\}_{n \geq 0}$  converges in ${W}_2$ to an absolutely continuous measure in $\mathcal{W}_{2}(\mathbb{R}^q)$ which is also a Karcher mean. If some $m_i$ has a bounded density, and if there exists a unique Karcher mean, then $\hat{m}$ is the Wasserstein barycenter of $\Pi$ and we have that $W_2(\mu_n, \hat{m}) \rightarrow 0$.
\end{proposition}

For the reader's convenience, we present next a counterexample to the uniqueness of Karcher means:
\begin{example}\label{ex:non_unique_Karcher}
	In $\mathbb R^2$ we take $\mu_1$ as the uniform measure on $B((-1,M),\epsilon)\cup B((1,-M),\epsilon)$ and $\mu_2$ the uniform measure on $B((-1,-M),\epsilon)\cup B((1,M),\epsilon)$, with $\epsilon$ a small radius and $M>>\epsilon$. Then, if $\Pi=\frac{1}{2}(\delta_{\mu_1}+ \delta_{\mu_2})$,   the uniform measure on $B((-1,0),\epsilon)\cup B((1,0),\epsilon)$ and  the uniform measure on   $B((0,M),\epsilon)\cup B((0,-M),\epsilon)$ are two distinct Karcher means.
\end{example}

Thanks to the \emph{Riemannian-like} geometry of $\mathcal{W}_{2,ac}(\mathbb{R}^q)$ one can reinterpret the iterations in eq.~ \eqref{eq:fixed-point} as a gradient descent step. This was discovered by Panaretos and Zemel in \cite{panaretos2017frechet,PaZe20}. In fact, in \cite[Theorem 1]{panaretos2017frechet} the authors  show that  {\color{black}
the functional  
\begin{align}\label{eq:FGD}
F(m) :=& \textstyle  \frac{1}{2}\sum_{i=1}^{L}\lambda_iW_2^2(m_i,m),
\end{align}
defined on $\mathcal{W}_2(\R^q)$, has 
 Fr\'echet derivative at each point $m\in  \mathcal{W}_{2,ac}(\R^q) $, given by
\begin{align}
\label{eq:Frechet-derivative}
F^\prime(m) =& \textstyle  -\sum_{i=1}^{L}\lambda_i(T_m^{m_i}-I) = I-\sum_{i=1}^{L}\lambda_iT_m^{m_i} \in L^2(m),
\end{align}
where $I$ is the identity map in $\mathbb R^q$. 
  More precisely, for such $m$ one has  when $W_2(\hat{m},m)$ goes to zero that
  \begin{equation}\label{eq:FrechDeriv}
  \frac{F(\hat{m})-F(m)-\int_{\R^q} \langle  F'(m)(x), T_m^{\hat{m}} (x)- x \rangle m(dx)  }{W_2(\hat{m},m)} \longrightarrow 0 , 
  \end{equation}
  thanks to \cite[Corollary 10.2.7]{ags08}. 
It follows from Brenier's theorem \cite[Theorem 2.12(ii)]{villani2003topics} that $\hat{m}$ is a fixed point of $G$ defined in eq.~\eqref{eq:G-operator} if and only if $F^\prime(\hat{m}) = 0$. The gradient descent sequence with step $\gamma$ starting from $\mu_0 \in \mathcal{W}_{2,ac}(\mathbb{R}^q)$ is then  defined by (c.f.\ \cite{panaretos2017frechet})
\begin{align}
\label{eq:gradient-descent}
\mu_{n+1}:= G_{\gamma}(\mu_n), \text{ for } n \geq 0,
\end{align}
where 
\begin{align*}G_{\gamma}(m) &:=  \textstyle \left[I + \gamma F^\prime(m)\right](m)= \left[(1-\gamma)I + \gamma \sum_{i=1}^{L}\lambda_iT_m^{m_i}\right](m)= \left[I + \gamma \sum_{i=1}^{L}\lambda_i(T_m^{m_i}-I)\right](m).
\end{align*}
Note that, by \eqref{eq:Frechet-derivative}, the iterations \eqref{eq:gradient-descent}  truly correspond to a gradient descent in $\mathcal{W}_2(\R^q)$ for the function in eq.~\eqref{eq:FGD}.}  We remark also  that, if $\gamma=1$,  the  sequence in eq.~\eqref{eq:gradient-descent} coincides with that in eq.~\eqref{eq:fixed-point}, i.e., $G_{1} = G$. These ideas serve us as inspiration for the stochastic gradient descent iteration in the next part.

\section{Stochastic gradient descent for barycenters in Wasserstein space}
\label{sec stoch grad}
The method presented in Section \ref{sec grad desc} is well suited for calculating the empirical barycenter. For the estimation of a population barycenter (i.e.\ when $\Pi$ does not have finite support) we would need to construct a convergent sequence of empirical barycenters, which can be computationally expensive. Furthermore, if a new sample from $\Pi$ arrives, the previous method would need to recalculate the barycenter from scratch. To address these challenges, we follow the ideas of stochastic algorithms \cite{RoMo51}, widely adopted in machine learning   \cite{bottou2018optimization},  and define a \emph{stochastic} version of the gradient descent sequence for the barycenter of $\Pi$. 

{\color{black}Recall that for $\mu_0 \in K_\Pi$ (in particular, $\mu_0$ absolutely continuous), $m_k \stackrel{\text{iid}}{\sim}  \Pi$ defined in some probability space $(\Omega, {\cal F}, \P)$  and $\gamma_k > 0$ for $k \geq 0$, we constructed the stochastic gradient descent (SGD) sequence as} 
	\begin{align*}
	\mu_{k+1} := \left[(1-\gamma_k)I + \gamma_k T_{\mu_k}^{m_k}\right](\mu_k) \text{ , for } k \geq 0.
	\end{align*}
%
The key ingredients for the convergence analysis of the above SGD iterations are the functions:
\begin{align}
F(\mu) :=& \textstyle  \frac{1}{2}\int_{\mathcal P(\mathcal X)} W_{2}^2(\mu,m)\Pi(dm),\label{eq:FSGD}\\
F^\prime(\mu) (x):=& \textstyle  -\int_{\mathcal P(\mathcal X)} (T_\mu^{m}-I))\Pi(dm)(x), \label{eq:F'SGD}
\end{align}
the natural analogues to the eponymous objects {\color{black} \eqref{eq:FGD} and  \eqref{eq:Frechet-derivative}.  In this setting, one can  formally (or rigorously, under additional assumptions) check that $F'$ is the actual Frechet derivative of $F$, and this  justifies naming  $\{\mu_k\}_k$ a SGD sequence.  In the sequel, the notation $F$ and $F'$ always refers to the functions defined in eqs.~\eqref{eq:FSGD} and \eqref{eq:F'SGD}, respectively. }

Observe that the population barycenter $\hat \mu$ is the unique minimizer of $F$. The following lemma justifies that $F'$ is well defined and that $\|F'(\hat\mu)\|_{L^2(\hat \mu)}=0$, and in particular $\hat \mu$ is a Karcher mean. This is a generalization of the corresponding result in \cite{alvarez2016fixed} where only the case $|\text{supp}(\Pi)|<\infty$ is covered. 

\begin{lemma}
	\label{measurability fixed point} {\color{black}Let $\tilde\Pi$ be a probability measure concentrated on $\mathcal W_{2,ac}(\mathbb R^q)$.  
There exists a jointly measurable function  $\mathcal W_{2,ac}(\mathbb R^q)\times\mathcal W_2(\mathbb R^q)\times \mathbb R^q \ni(\mu,m,x)\mapsto T^m_\mu(x)$  which is $\mu(dx)\Pi(dm)\tilde\Pi(d\mu)$}-a.s.\ equal to the  unique optimal transport map from $\mu$ to $m $ at $x$. Furthermore, letting $\hat \mu$ be a barycenter of $\Pi$, we have
	$x \,=  \, \int T^m_{\hat\mu}(x) \Pi(dm),\,\, \hat \mu(dx)$-a.s.
\end{lemma}

\begin{proof}
		The existence of a jointly measurable version of the unique optimal maps is proved in  \cite{FGM}. Let us prove the last assertion. Letting $T^m_{\hat\mu}=:T^m$, we have by Brenier's theorem \cite[Theorem 2.12(ii)]{villani2003topics}  that
		\begin{align*}
		 \textstyle   \int W_2(\hat \mu,m)^2\Pi(dm) =  &  \textstyle \int\int |x-T^m(x)|^2\hat \mu(dx)\Pi(dm) \\
    = & \textstyle \int\int |x-  \int T^{\bar m}(x)\Pi(d{\bar m})   +  \int T^{\bar m}(x)\Pi(d{\bar m}) - T^m(x)|^2\hat \mu(dx)\Pi(dm) \\
     = &  \textstyle  \int \left| x- \textstyle \int T^{\bar m}(x)\Pi(d{\bar m}) \right|^2 \hat \mu(dx) +  \textstyle \int\int |  \int T^{\bar m}(x)\Pi(d{\bar m}) - T^m(x)|^2\hat \mu(dx)\Pi(dm) 
		\end{align*}
  where we used the fact that $2 \int \int \langle x - \int T^{\bar m}(x)\Pi(d{\bar m})  ,  \int T^{\bar m}(x)\Pi(d{\bar m}) 
 - T^m(x) \rangle \Pi(dm)\hat \mu(dx) =0 $.
 The second term in the last line is an upper bound for 
 $$\textstyle \int W_2\left(\left( \int T^{\bar m} \Pi(d\bar{m})\right )( \mu) \,,\, m\right)^2\Pi(d m) \geq  \int W_2(\hat \mu,m)^2\Pi(dm) .$$
 We conclude that $  \int \left| x- \textstyle \int T^{\bar m}(x)\Pi(d{\bar m}) \right|^2 \hat \mu(dx)$
 as required. 
		
	\end{proof}

\medskip 

{\color{black} Notice that  Lemma \ref{measurability fixed point} ensures that the SGD  sequence $\{\mu_k\}_k$ is  well defined as a  sequence of (measurable) $\mathcal{W}_2$-valued random variables.  More precisely, denoting by $\mathcal F_{0}$ the trivial sigma-algebra and $\mathcal F_{k+1},k\geq 0,$  the sigma-algebra generated by $m_0,\dots,m_{k}$,  one can inductively apply the first part of Lemma \ref{measurability fixed point} with $\tilde{\Pi}=law (\mu_{k})$ to check that $T_{\mu_k}^{m_k}(x)$ is measurable with respect to $\mathcal F_{k+1}\otimes \mathcal B (\R^q)$, where $\mathcal B$ stands for the Borel sigma-field. This implies that both $\mu_{k}$ and  $\lVert F^\prime(\mu_{k})\rVert^2_{L^2(\mu_{k})}$  are  measurable w.r.t. $\mathcal F_{k}$. }

{\color{black} The next proposition suggests that, in expectation, the sequence $\{F(\mu_k)\}_k$ is essentially decreasing. This is a first insight into the behaviour of the sequence $\{\mu_k\}_k$.}



\begin{proposition}
	\label{prop:bound-step}
	The SGD sequence in eq.~\eqref{eq:sgd-seq} verifies (a.s.) that 
	\begin{align}
	\label{eq:sgd-ineq}
	\mathbb{E}\left[F(\mu_{k+1}) - F(\mu_k)|\mathcal F_{k}\right] \leq \gamma_k^2F(\mu_k) - \gamma_k\lVert F^\prime(\mu_k)\rVert^2_{L^2(\mu_k)}.
	\end{align}
	\begin{proof}
		Let  $\nu \in\text{supp}( \Pi)$. Clearly
		$ \textstyle  \left( \left[(1-\gamma_k)I + \gamma_k T_{\mu_k}^{m_k}\right]\, ,\,T_{\mu_k}^\nu \right )(\mu_k) $
		is a feasible (not necessarily optimal) coupling with first and second marginals $\mu_{k+1}$ and $\nu$ respectively.  Denoting $O_m:= T_{\mu_k}^{m} -  I $,  we have
		\begin{align*}
		W_2^2(\mu_{k+1},\nu) 	  	\leq \lVert (1-\gamma_k)I + \gamma_k T_{\mu_k}^{m_k} - T_{\mu_k}^{\nu}  \rVert^2_{L^2(\mu_k)}
		&= \lVert - O_{\nu} + \gamma_k O_{m_k} \rVert^2_{L^2(\mu_k)}\\
		&= \lVert O_{\nu} \rVert^2_{L^2(\mu_k)} -2\gamma_k\langle O_{\nu}, O_{m_k}\rangle_{L^2(\mu_k)} +  \gamma_k^2\lVert O_{m_k} \rVert^2_{L^2(\mu_k)}.
		\end{align*}
		Evaluating $\mu_{k+1}$ on the functional $F$ and thanks to the previous inequality, we have
		\begin{align*}
		F(\mu_{k+1}) =&  \textstyle \frac{1}{2}\int W_{2}^2(\mu_{k+1},\nu)\Pi(d\nu)\\
		\leq& \textstyle  \frac{1}{2}\int\lVert O_{\nu} \rVert^2_{L^2(\mu_k)}\Pi(d\nu) -\gamma_k\left\langle \int O_{\nu}\Pi(d\nu), O_{m_k}\right\rangle_{L^2(\mu_k)} +  \frac{\gamma_k^2}{2}\lVert O_{m_k} \rVert^2_{L^2(\mu_k)}\\
		=& \textstyle F(\mu_k)+\gamma_k\left\langle F^\prime(\mu_k), O_{m_k}\right\rangle_{L^2(\mu_k)} +  \frac{\gamma_k^2}{2}\lVert O_{m_k} \rVert^2_{L^2(\mu_k)}.
		\end{align*}
		Taking conditional expectation with respect to $\mathcal F_k$, and as $m_k$ is independently sampled from this sigma-algebra, we conclude
		\begin{align*}
		\mathbb{E}\left[F(\mu_{k+1})|\mathcal F_k\right] \leq &  \textstyle F(\mu_k)+\gamma_k\left\langle  F^\prime(\mu_k), \int O_{m}\Pi(dm)\right\rangle_{L^2(\mu_k)} +  \frac{\gamma_k^2}{2}\int\lVert O_{m} \rVert^2_{L^2(\mu_k)}\Pi(dm)\\
		=&    (1+\gamma_k^2)F(\mu_k) - \gamma_k\lVert F^\prime(\mu_k) \rVert^2_{L^2(\mu_k)}.
		\end{align*}
	\end{proof}
\end{proposition}

Next Lemma states key continuity properties of the functions  $F$ and $F'$.  {\color{black}For this result,  assumption (A2) can be dropped and it is only required that  $\int\int\|x\|^2m(dx)\Pi(dm)<\infty$.} 

\begin{lemma}
	\label{continuity norm F' }
	 Let $(\rho_n)_n \subset \mathcal{W}_{2,ac}(\RR^q)   $  be a sequence converging w.r.t.\  $W_2$ to $\rho\in  \mathcal{W}_{2,ac}(\RR^q) $. Then, as $n\to \infty$,we have \textcolor{black}{ i) :  $F(\rho_n)\to  F(\rho)$ and ii) : 
	  $\| F'(\rho_n)\|_{L^2(\rho_n)}\to \| F'(\rho)\|_{L^2(\rho)}$.}
\end{lemma}
\begin{proof} 
 We prove both convergence claims using Skorokhod's representation theorem. Thanks to that result, in a given probability space $(\Omega,{\cal G}, \P)$ one can simultaneously construct a sequence of random vectors $(X_n)_n$ of laws $(\rho_n)_n $, and a random variable $X$ of law $\rho$, such that $(X_n)_n$ converges $\P-$ a.s. to $X$.  Moreover,  by Theorem 3.4 in \cite{cuestaalbertos1997optimal}, the sequence  $(T_{\rho_n}^m (X_n))_n$ converges $\P-$ a.s. to $T_{\rho}^m (X)$. Notice that, for all $n\in \NN$,   $T_{\rho_n}^m (X_n)$ distributes according to the law $m$, and the same holds true for $T_{\rho}^m (X)$.

 We now enlarge the probability space  $(\Omega,{\cal G}, \P)$ (maintaining the same notation for simplicity)  with an independent  random variable  $\m$  in   $ \mathcal{W}_2(\R^d)$ with law  $\Pi$ (thus independent of  $(X_n)_n$  and $X$). {\color{black}  Applying Lemma   \ref{measurability fixed point} with $\tilde{\Pi}=\delta_{\rho_n}$ for each $n$, or with $\tilde{\Pi}=\delta_{\rho}$,  one can show that  $ (X_n, T_{\rho_n}^{\m}  (X_n) ) $  and  $ (X, T_{\rho}^{\m} (X) ) $ are random variables  in  $(\Omega,{\cal G}, \P)$. By conditioning on $\{ \m=m\}$ one further obtains that 
 $ (X_n, T_{\rho_n}^{\m}  (X_n) )_{n\in \NN}$  converges  $\P-$ a.s.~to  $ (X, T_{\rho}^{\m} (X) ) $.}

Notice that  $\sup_n \E \left( \| X_n\|^2 \mathbf{1}_{ \|X_n\|^2 \geq M}\right) = \sup_n \int_{ \|x\|^2\geq M}  \|x\|^2 \rho_n(dx) \to 0$ as $M\to \infty$  since $(\rho_n)_n$ converges in $ \mathcal{W}_2(\mathcal{X})$, while, upon conditioning on $\m$, 
$$\textstyle  \sup_n \E \left( \|  T_{\rho_n}^{\m}  (X_n) \|^2 \mathbf{1}_{ \| T_{\rho_n}^{\m}  (X_n) \|^2 \geq M}\right)=   \int_{\mathcal{W}_2(\mathcal{X})}\left(  \int_{ \|y\|^2\geq M}  \|y\|^2 m (dy)  \right) \Pi(dm)\to 0   $$
by dominated convergence, since $\int_{ \|x\|^2\geq M}  \|x\|^2 m (dx) \leq \int  \|x\|^2 m (dx) = W_2^2(m,\delta_0) $ and $\Pi\in \mathcal{W}_2( \mathcal{W}_2(\mathcal{X}))$. By Vitalli's convergence Theorem, we deduce that $ (X_n, T_{\rho_n}^{\m}  (X_n) )_{n\in \NN}$  converges  to  $ (X, T_{\rho}^{\m} (X) ) $ in $L^2 (\Omega,{\cal G}, \P)$.  \textcolor{black}{In particular, as $n\to \infty$ we have
$$ \textstyle \int_{\mathcal P(\mathcal X)} W_{2}^2(\rho_n,m)\Pi(dm) =\E |X_n - T_{\rho_n}^{\m}  (X_n) |^2\to \E |X - T_{\rho}^{\m}  (X_n) |^2  =  \int_{\mathcal P(\mathcal X)} W_{2}^2(\rho,m)\Pi(dm)  ,$$
which proves  convergence i).}  Denoting  now by  ${\cal G}_{\infty}$  the sigma-field generated by  $(X_1,X_2,\dots )$ we also obtain that
\begin{equation}\label{conv cond expect}  \E ( X_n-   T_{\rho_n}^{\m}  (X_n) \vert {\cal G}_{\infty}   )   \to      \E ( X-   T_{\rho}^{\m}  (X) \vert {\cal G}_{\infty}  )   \mbox{ in } L^2 (\Omega,{\cal G}, \P). 
\end{equation}
  Observe now that the following identities hold:
 \begin{equation*}
\begin{split}
F'(\rho_n) (X_n)   =   & \,  \E  (X_n-   T_{\rho_n}^{\m}  (X_n) \vert X_n ) =    \E ( X_n-   T_{\rho_n}^{\m}  (X_n) \vert   {\cal G}_{\infty}  ) \,   \\
F'(\rho) (X)   =   & \,  \E ( X-   T_{\rho}^{\m}  (X) \vert X ) =    \E ( X-   T_{\rho}^{\m}  (X) \vert   {\cal G}_{\infty}  ). \\
\end{split}
\end{equation*} 
 The convergence ii) follows   from \eqref{conv cond expect}, that $  \E ( F'(\rho_n) (X_n))^2 =  \| F'(\rho_n)\|_{L^2(\rho_n)}$, and  $  \E ( F'(\rho) (X))^2  = \| F'(\rho)\|_{L^2(\rho)}  $.

\end{proof}

We now proceed to prove the first of our two main results.

\begin{proof}[Proof of Theorem \ref{thm:sgd-convergence}]
	Let us denote $\hat \mu$ the unique barycenter, write $\hat{F} := F(\hat{\mu})$ and introduce 
	\begin{equation*}
	 h_t := F(\mu_t) - \hat{F}\geq  0,\,\,\,\,\mbox{  and }\,\,\, \textstyle
	\alpha_t := \prod_{i=1}^{t-1}\frac{1}{1+\gamma_i^2}.
	\end{equation*}
	 Thanks to the condition in eq.~\eqref{eq:cond1-gamma} the sequence $(\alpha_t)$ converges to some finite $\alpha_\infty >0$, as it can be verified simply by applying logarithm. By Proposition \ref{prop:bound-step} we have
	\begin{align}
	\label{eq:h-diff-bound}
	 \mathbb{E}\left[h_{t+1} - (1+\gamma_t^2)h_t|\mathcal F_t\right] \leq \gamma_t^2\hat{F}-\gamma_t\lVert F^\prime(\mu_t)\rVert^2_{L^2(\mu_t)} \leq \gamma_t^2\hat{F} ,
	\end{align} 
	from where,  after multiplying by $\alpha_{t+1}$, the  following bound is derived: 
	\begin{align}
	\mathbb{E}\left[\alpha_{t+1}h_{t+1} - \alpha_t h_t |\mathcal F_t\right] &\leq \alpha_{t+1} \gamma_t^2\hat{F} -\alpha_{t+1}\gamma_t\lVert F^\prime(\mu_t)\rVert^2_{L^2(\mu_t)} \leq \alpha_{t+1} \gamma_t^2\hat{F}.\label{eq alpha h important}
	\end{align}
	Defining now $\hat{h}_t:= \alpha_t h_t - \sum_{i=1}^{t} \alpha_i \gamma_{i-1}^2  \hat{F}$, we deduce from eq.~\eqref{eq alpha h important} that
	$ \mathbb{E}\left[\hat{h}_{t+1} -\hat{h}_t |\mathcal F_t\right] \leq 0$, namely  that $(\hat{h}_t)_{t\geq 0}$ is a supermartingale w.r.t 
$(\mathcal F_t)$. {\color{black}The fact that $(\alpha_t)$ is convergent, together with the condition in eq.~\eqref{eq:cond1-gamma},   ensures that $\sum_{i=1}^{\infty} \alpha_i \gamma_{i-1}^2  \hat{F}<\infty$, thus $\hat{h}_t$ is uniformly lower bounded by a constant. Therefore, the supermartingale} convergence theorem \cite[Corollary 11.7]{Wil} implies the existence of $\hat{h}_{\infty}\in L^1$ such that $\hat{h}_t\to \hat{h}_{\infty}$ a.s. But then, necessarily, $h_t\to h_{\infty}$ a.s.\  for some nonnegative random variable $h_{\infty}\in L^1$.   
	
	Thus, our goal now is to prove that  $h_{\infty}=0$ a.s.  Taking expectations in eq.~\eqref{eq alpha h important} and summing over $t$ to obtain a telescopic summation, we obtain
	$$\textstyle \mathbb E[\alpha_{t+1}h_{t+1}]- \mathbb E[ h_0\alpha_0 ]  \leq \hat F \sum_{s=1}^t \alpha_{s+1}\gamma_s^2 - \sum_{s=1}^t \alpha_{s+1}\gamma_s \E\left[  \lVert F^\prime(\mu_s)\rVert^2_{L^2(\mu_s)}\right].$$
	Then, taking liminf, applying Fatou on the l.h.s.~and monotone convergence on the r.h.s.\ we obtain: 
	$$\textstyle -\infty < \mathbb E[\alpha_{\infty}h_{\infty}]- \E[ h_0\alpha_0 ]\leq C- \mathbb E\left[\sum_{s=1}^\infty \alpha_{s+1}\gamma_s\lVert F^\prime(\mu_s)\rVert^2_{L^2(\mu_s)}\right ].$$ 
	 In particular,  since $(\alpha_t)$ is bounded  away from $0$, we have
	\begin{align}\label{eq sum product} \textstyle
	  \sum_{t=1}^{\infty} \gamma_t\lVert F^\prime(\mu_t)\rVert^2_{L^2(\mu_t)} < \infty \,\,\,\text{a.s.}
	\end{align}
	Note that $\P( \liminf_{t\to \infty} \lVert F^\prime(\mu_t)\rVert^2_{L^2(\mu_t)}>0)>0$ would be at odds with the conditions in eqs.~\eqref{eq sum product} and \eqref{eq:cond2-gamma}, so
	\begin{equation}\label{eq:liminf0} \textstyle
	\liminf_{t\to \infty} \lVert F^\prime(\mu_t)\rVert^2_{L^2(\mu_t)}=0,\,\, \text{a.s.}
\end{equation}
	{\color{black} Observe also that, from Assumption (A2) and Lemma \ref{continuity norm F' }, we have 
	\begin{equation}\label{eq:infepsK} \textstyle 
	\forall \varepsilon>0 , \, \inf_{\{ \rho : F(\rho) \geq \hat F +\varepsilon \} \cap K_\Pi}  \,  \lVert F^\prime(\rho)\rVert^2_{L^2(\rho)}>0.  
	\end{equation}
	 Indeed, one can see  that the  set $\{ \rho : F(\rho) \geq \hat F +\varepsilon \}\cap K_\Pi $ is  $W_2$-compact, using part i) of Lemma \ref{continuity norm F' }, and then check that the function $\rho\mapsto \lVert F^\prime(\rho)\rVert^2_{L^2(\rho)}$ attains its minimun on it, using part ii) of that result. That minimum cannot be zero, as otherwise we would have obtained a Karcher mean that is not equal to the barycenter (contradicting the uniqueness of the Karcher mean). Remark also that a.s.\ $\mu_t\in K_\Pi $ for each $t$, by the geodesic convexity part of Assumption (A2). We deduce the  a.s.\ relationships between  events:
	\begin{equation*}
	\begin{split}
	\{ h_{\infty} \geq  2\varepsilon \} \subset \, &  \textstyle\left\{ \mu_t \in   \{ \rho : F(\rho) \geq \hat F +\varepsilon \} \cap K_\Pi  
	\, \forall t \mbox{ large enough} \right\} \\
	  \subset \, & \textstyle \bigcup_{\ell\in \N}     \left\{  \lVert F^\prime(\mu_t)\rVert^2_{L^2(\mu_t)} > 1/\ell :\, \forall  t \mbox{ large enough} \right\} 
	  \subset \, \textstyle \left\{ \liminf_{t\to \infty} \lVert F^\prime(\mu_t)\rVert^2_{L^2(\mu_t)}>0\right\}  ,
	 \end{split}
	 \end{equation*}}where eq.~\eqref{eq:infepsK} was used to obtain the second inclusion. It follows using  eq.~\eqref{eq:liminf0}    that $\P(h_{\infty} \geq  2\varepsilon)=0$ for every $\varepsilon>0$, hence $h_{\infty}=0$ as required.
	To conclude,  we use the fact that sequence $\{\mu_t\}_t$ is a.s.\ contained in the  $W_2$-compact  $ K_\Pi$, {\color{black} by assumption (A2), and  the first convergence in Lemma \ref{continuity norm F' }, to deduce that  the limit $\tilde{\mu}$ of any convergent subsequence  satisfies $F(\tilde{\mu})-\hat F=h_{\infty}=0$, and then $F(\tilde\mu)=\hat F$.} Hence $\tilde{\mu}= \hat{\mu}$ by uniqueness of the barycenter. This implies that  $\mu_t\to \hat{\mu}$  in $\mathcal{W}_2(\mathcal{X})$ a.s. as $t\to \infty$.
\end{proof}

\begin{remark}\label{rem:variance}
 Observe that, for  a fixed $\mu\in \mathcal{W}_{2,ac}(\mathcal{X})$ and a random $m\sim \Pi$, the random variable  $-(T_{\mu}^{m}-I)(x)$  is  an unbiased estimator of $F^\prime(\mu)(x)$,  for  $\mu(dx)$ a.e.\ $x\in \R^q$. A natural way to jointly quantify the pointwise variances of these  estimators is through the {\it integrated variance}:  
\begin{equation}\label{eq:avvar}
\mathbb{V}[-(T_{\mu}^{m}-I)]: = \int\text{Var}_{m\sim\Pi}[T_{\mu}^{m}(x)- x] \mu(dx),
\end{equation} 
which is the equivalent (for unbiased estimators) of the mean integrated square error (MISE) from non-parametric statistics \cite{wasserman2006all}. Simple computations yield the following expresion for it, which will be useful in the next two sections: 
		\begin{equation}\label{eq:VE}
			\textstyle \mathbb{V}[-(T_{\mu}^{m}-I)] =\textstyle \mathbb{E}\left[ \lVert -(T_{\mu}^{m}-I) \rVert_{L^2(\mu)}^2 \right] -   \left\lVert\mathbb{E}\left[ -(T_{\mu}^{m}-I) \right]\right\rVert_{L^2(\mu)}^2 = 2F(\mu) -  \lVert  F^\prime(\mu) \rVert_{L^2(\mu)}^2.
		\end{equation}
\end{remark}

\smallskip

We close this section with the promised statement of Lemma \ref{lem:verification_A2}, referred to in the introduction, giving us  a sufficient condition for Assumption (A2):

\begin{lemma}
    \label{lem:verification_A2}
    If the support of $\Pi $ is $\mathcal W_2(\mathbb R^d)$-compact and there is a constant $C_1<\infty$ such
     that $$\textstyle\Pi(\{m\in \mathcal W_{2,ac}(\mathbb R^d): \int \log(dm/dx)m(dx) \leq C_1\})=1,$$  then Assumption (A2) is fulfilled.
\end{lemma}

\begin{proof}
Let $K$ be the support of $\Pi$, which is $\mathcal W_2(\mathbb R^d)$-compact. By de la Vall\'ee Poussin criterion, there is $V:\mathbb R_+\to \mathbb R_+$ increasing, convex and super-quadratic (i.e.\ $\lim_{r\to +\infty} V(r)/r^2=+\infty$) such that $C_2:=\sup_{m\in K}\int V(\|x\|)m(dx)<\infty$. Observe that $p(dx):= \exp\{-V(\|x\|)\}dx$ is a finite measure (wlog a probability measure). Moreover, for the relative entropy with respect to $p$ we have
$H(m|p):=\int \log(dm/dp)dm= \int  \log(dm/dx)m(dx) + \int V(\|x\|)m(dx) $ if $m\ll dx$ and $+\infty $ otherwise. Define now $K_\Pi:=\{m\in \mathcal W_{2,ac}(\mathbb R^d): H(m|p)\leq C_1+C_2,\, \int V(\|x\|)m(dx) \leq C_2 \}$ so that $\Pi(K_\Pi)=1$. Clearly $K_\Pi$ is $\mathcal W_2$-closed, since the relative entropy is weakly lower semicontinuous, and also $\mathcal W_2$-relatively compact, since $V$ is super-quadratic. Finally $K_\Pi$ is geodesically convex by \cite[Theorem 5.15]{villani2003topics}.
\end{proof}

{For instance if all $m$ in the support of $\Pi$ is of the form $m(dx)=\frac{e^{-V_m(x)}}{\int_y e^{-V_m}(y)dy }dx$ with $V_m$ bounded from below,  then one way to guarantee the conditions in Lemma \ref{lem:verification_A2}, assuming w.l.o.g.\ that $V_m\geq 0$,  is to ask that $\int_y e^{-V_m}(y)dy \geq A $ and $\int_y |y|^{2+\epsilon}e^{-V_m}(y)dy\leq B$, for some  fixed $\epsilon,A,B >0$. In words: tails are controlled and the measures cannot be too concentrated. }

\section{A condition granting uniqueness of Karcher means and convergence rates }
\label{sec unique Karcher SGD rates}
The aim of this  section is to prove Theorem \ref{thm:uKarcherSGDrate},  a  refinement of Theorem \ref{thm:sgd-convergence} under the additional assumption that the barycenter is a pseudo-associative Karcher mean.  Notice that,  with the notation introduced in Section \ref{sec stoch grad},  Definition \ref{def:pseudo-asoc} of  pseudo-associative Karcher mean $\mu$ simply reads as
\begin{equation}\label{eq:WCF}
    W_2^2(\mu, \nu)\leq C_{\mu} \lVert F^\prime(\nu) \rVert^2_{L^2(\nu)}
\end{equation}
for all $\nu  \in  \mathcal W_{2,ac}({\cal X})$.

\begin{remark}\label{rem:assocKarcher}
    Suppose $\mu$ is a Karcher mean of $\Pi$ such that, for all $\nu\in \mathcal W_{2,ac}({\cal X}) $ and $\Pi(dm)$ a.e. $m$ one has 
\begin{equation}\label{eq:associativeTTT} T_{\mu}^{m}(x) = T_{\nu}^m \circ T_{\mu}^{\nu} (x), \quad \mu(dx)\,  \mbox{ 
  a.e. } x   \in \RR^q.
  \end{equation}
  Then, $\mu$ is pseudo-associative, with $C_{\mu}=1$ and equality holding in   Definition  \ref{def:pseudo-asoc}. Indeed, in that case we have: 
  \begin{equation}\label{eq:asocimplypseudo}
  \begin{split}
     \textstyle \lVert F^\prime(\nu) \rVert^2_{L^2(\nu)} = &  \textstyle  \,   \int_{{\cal X}} \left |     \int_{\mathcal P(\mathcal X)} ( T_{\nu}^m \circ T_{\mu}^{\nu} (x)-T_{\mu}^{\nu} (x) )\Pi(dm) \right |^2 \mu(dx)  \\ 
     = &  \textstyle  \,   \int_{{\cal X}} \left |       \int_{\mathcal P(\mathcal X)} 
 T_{\mu}^{m} (x) \Pi(dm) -T_{\mu}^{\nu} (x)  \right |^2 \mu(dx)  \\ 
 = & \,   \textstyle   \int_{{\cal X}} \left |      x- T_{\mu}^{\nu} (x)  \right |^2 \mu(dx) = W_2^2(\mu,\nu) \\ 
       \end{split}
  \end{equation}
We thus will say that the Karcher mean $\mu$ is {\it associative} if simply   \eqref{eq:associativeTTT}  holds. 
\end{remark}

The following is an  analogue in Wasserstein space, of a classical property implying convergence rates in gradient-type optimization algorithms:

\begin{definition} We say that $\mu \in   \mathcal W_{2,ac}({\cal X})$ satisfies a Polyak-Lojasiewicz inequality if 
    \begin{equation}\label{eq:PL} \textstyle 
  F(\nu)-F(\mu)\leq \frac{\bar{C}_{\mu}}{2}  \lVert F^\prime(\nu) \rVert^2_{L^2(\nu)}
\end{equation}
for some $\bar{C}_{\mu}>0$ and every  $\nu \in   \mathcal W_{2,ac}({\cal X})$.
\end{definition}

We next state some useful properties:

\begin{lemma}\label{lemma:inequalities}
Suppose $\mu$ is a Karcher mean of $\Pi$. Then:
\begin{itemize}
    \item[i)]For all $\nu \in  \mathcal W_{2,ac}({\cal X})$    we have
\begin{equation}\label{FFW}  \textstyle 
    F(\nu)-F(\mu)\leq \frac{1}{2} W_2^2(\mu,\nu). 
\end{equation}
\item[ii)] If $\mu$ is pseudo-associative, then it is the unique barycenter, and it satisfies the Polyak- Lojasiewicz inequality \eqref{eq:PL} with $\bar{C}_{\mu}=C_{\mu}$. 
\item[iii)] If the associativity relation \eqref{eq:associativeTTT} holds, then we have
 \begin{equation*} \textstyle  F(\nu)-F(\mu)=\frac{1}{2} W_2^2(\mu,\nu)=  \frac{1}{2}  \lVert F^\prime(\nu) \rVert^2_{L^2(\nu)} .  
 \end{equation*}
 \end{itemize}
\end{lemma}

\begin{proof}
    i) { We notice that this is a particular case of Theorem 7 in \cite{chewi2020gradient},  which we can prove by an elementary argument based on the notion of Karcher mean. Indeed, by definition of  function $F$  and the fact that} $(T_{\mu}^{\nu}, T_{\mu}^m)( \mu)$ is a coupling of $\nu$ and $m$ we have
    \begin{equation*}
        \begin{split}
            F(\nu)-F(\mu)\leq & \textstyle  \, \frac{1}{2} \int_{{\cal P}({\cal X}) } \int_{{\cal X}} \left\{ |T_{\mu}^{\nu}(x)-T_{\mu}^m(x) |^2- |x-T_{\mu}^m(x) |^2   \right\} \mu(dx) \Pi(dm) \\ 
        = & \textstyle  \, \frac{1}{2} \int_{{\cal P}({\cal X}) } \int_{{\cal X}} \left\{|T_{\mu}^{\nu}(x)|^2 -2 \langle T_{\mu}^m(x) ,T_{\mu}^{\nu}(x)\rangle - x^2+2 \langle x, T_{\mu}^m(x) \rangle \right\} \mu(dx) \Pi(dm) \\
        = & \textstyle  \, \frac{1}{2}  \int_{{\cal X}} |  T_{\mu}^{\nu}(x) - x|^2  \mu(dx) 
        \end{split}
    \end{equation*}
    where in the second equality we used twice  the fact that $\int_{{\cal P)}({\cal X}) }  T_{\mu}^{m}(x) \Pi(dm) = x$ for $\mu(dx)$ a.e.\ $x$. 
    
    ii) The claim is obvious in view of the previous and of \eqref{eq:WCF}. 
    
    iii) Taking into account Remark \ref{rem:assocKarcher}, it is enough to notice that 
    $$ \textstyle   F(\nu)= \int_{{\cal P}({\cal X}) } \int_{{\cal X}} |y -T_{\nu}^m(y) |^2\nu(dy) \Pi(dm)= \int_{{\cal P}({\cal X}) } \int_{{\cal X}} |T_{\mu}^{\nu}(x)-T_{\mu}^m(x) |^2\mu(dx) \Pi(dm)   $$
    under \eqref{eq:associativeTTT},   in which case the inequality in the proof of i) is an equality. 
    
\end{proof}

\begin{remark}\label{rem:VI+PLassocKarcher} Recall that $\Pi$ is said to satisfy a variance inequality \cite{ACLGP} if for some constant $C>0$,  
$$ \textstyle  F(\nu)-F(\hat\mu)\geq \frac{C}{2} W_2^2(\hat\mu,\nu),$$
with $\hat{\mu}$ the barycenter of $\Pi$. It readily follows that $\hat\mu$ is a pseudo-associative Karcher mean as soon as a variance inequality and the Polyak- Lojasiewicz inequality
\ref{eq:PL} hold. That was the case in \cite{chewi2020gradient}, and so the barycenter in the Gaussian setting considered therein is  a pseudo-associative Karcher mean too. {Notice that if the barycenter is associative, by Lemma \ref{lemma:inequalities}.(iii)  the variance inequality holds; this is the case of the examples discussed in Sections \ref{Ex:1dim}-\ref{Ex:sphere} below.}   { Notice also that, if a variance inequality holds, bounds on the optimality gap for the SDG sequence as in Theorem \ref{thm:uKarcherSGDrate} immediately yield similar bounds for the expected squared Wasserstein distance $\mathbb{E} [W_2^2(\mu_k,\hat{\mu}) ]$. } 

\end{remark}

\begin{proof}[Proof of Theorem \ref{thm:uKarcherSGDrate}]
If we assume the pseudo-associtativity condition  \eqref{eq:pseudo-asoc_ineq}  holds, 
uniqueness of Karcher means (hence equal to the barycenter) is immediate,  as noted in Section \ref{sec:SGDW}.
Under that assumption, the inequality \eqref{eq:PL} holds  thanks to  part ii) of Lemma \ref{lemma:inequalities}  and convergence estimates can then be deduced adapting classic arguments of the SGD algorithm in an Euclidean setting, see e.g. \cite{BCN2018}. 
Indeed, by Proposition \ref{prop:bound-step} and formulae \eqref{eq:VE} and \eqref{eq:PL}, we get that  
\begin{equation*}
	\begin{split}
	    \mathbb{E}\left[F(\mu_{k+1}) - F(\mu_k)|\mathcal F_{k}\right] 
     \leq &  \, \gamma_k^2F(\mu_k) - \gamma_k\lVert F^\prime(\mu_k)\rVert^2_{L^2(\mu_k)} \\
      = & \frac{\gamma_k^2}{2} \textstyle \mathbb{V}[-(T_{\mu_k}^{m}-I)]  - (1- \frac{\gamma_k}{2}) \, \gamma_k \, \lVert F^\prime(\mu_k)\rVert^2_{L^2(\mu_k)} \\ 
      \leq &  
    \gamma_k^2  \bar{F}  -  (1- \frac{\gamma_k}{2}) \frac{2\gamma_k }{C_{\mu}}  (F(\mu_k) - \hat{F} )  \\
     \leq &  
    \gamma_k^2  \bar{F}  -   \gamma_k C_{\mu}^{-1}  (F(\mu_k) - \hat{F} )  
	\end{split}
	\end{equation*}
for some finite (deterministic) upper bound $\bar{F}>0 $   of the sequence $(F(\mu_t))_{t\geq 1}$  under (A2'). In the last line we used the fact that  $\gamma_k\leq \gamma_0\leq 1$  for the choice $\gamma_k=\frac{a}{b+k}$ with fixed $b\geq a>0$. 
It follows that 
\begin{equation}\label{eq:sgd-ineq_pseudo-asoc}
 \textstyle \mathbb{E}\left[F(\mu_{k+1}) - \hat{F} 
 \right] \leq  \gamma_k^2\bar{F} + (1-\gamma_kC_{\mu}^{-1}) \mathbb{E}\left[F(\mu_{k}) - \hat{F} \right] \quad  \mbox{ for all }k\in \NN. 
\end{equation}
 Assume now that $ a>C_{\hat \mu}$ and that, for certain $c\geq a^2 \bar{F}/(C_{\hat \mu}^{-1}a-1)>0$,  it holds for a given $k\in \NN$ that 
\begin{equation}\label{eq:hypinduc}
 \textstyle \mathbb{E}\left[F(\mu_{k}) - \hat{F} \right]\leq c/(b+k).
\end{equation}
Let us show then that   $\mathbb{E}\left[F(\mu_{k+1}) - \hat{F} \right]\leq c/(b+k+1)$ too. By \eqref{eq:sgd-ineq_pseudo-asoc} we have  
\begin{equation*}
\begin{split}
 \textstyle \mathbb{E}\left[F(\mu_{k+1}) - \hat{F} 
 \right] \leq &  \textstyle  \, c\frac{(b+k-1)}{(b+k)^2}+ \frac{c(1- aC_{\mu}^{-1})+ a^2 \bar{F}}{(b+k)^2} \\
 \leq & \textstyle  \,  c\frac{(b+k-1)}{(b+k)^2} \\
 \leq & \textstyle  \,  \frac{c}{b+k+1}, \\
 \end{split}
\end{equation*}
where we used the choice of $c$ in the second inequality. Taking $c=\max\left\{ b \,  \mathbb{E}\left[F(\mu_0) - \hat{F} \right] ,a^2 \bar{F}/(C_{\hat\mu}^{-1}a-1)\right\} $, inequality  \eqref{eq:hypinduc} holds for $k=0$, hence for all $k\in \NN$ by induction. 

{ Concerning the a.s.\ weak-convergence of $\mu_k$ to $\hat \mu$, we can follow the proof of Theorem \ref{thm:sgd-convergence} up to Equation \eqref{eq:liminf0}, and then derive in the present setting \eqref{eq:infepsK}, i.e. 
\begin{align*} \textstyle 
	\forall \varepsilon>0 , \, \inf_{\{ \rho : F(\rho) \geq \hat F +\varepsilon \} \cap K_\Pi}  \,  \lVert F^\prime(\rho)\rVert^2_{L^2(\rho)}>0,  
	\end{align*}
 from the Polyak- Lojasiewicz inequality \eqref{eq:PL}, which holds thanks to Lemma \ref{lemma:inequalities}, (ii). As before, this is then used to obtain that $F(\mu_k)\to \hat F$ almost surely. Since $K_\Pi$ is $\mathcal W_2$-bounded, it is in particular tight. The fact that  $\nu\mapsto F(\nu)$ is weakly l.s.c.  and the uniqueness of minimizers of $F$  entail now that a.s.\ $\mu_k\to\hat\mu$ weakly.}
\end{proof} 

{ 
\begin{remark} In addition to concluding that a.s.\ $\mu_k\to\hat\mu$ weakly, the above proof also establishes the a.s.\ existence of a subsequence $n_k$ such that $\mu_{n_k}\to\hat\mu$ in $\mathcal W_2$. Indeed, this follows from \eqref{eq:liminf0} together with \eqref{eq:WCF}.   
\end{remark}}

 The above proof shows that we may relax the definition of pseudo-associativity by requiring \eqref{eq:pseudo-asoc_ineq} to hold for $\nu\in K_\Pi$ (or more precisely, for $\nu\in\{\mu_k:k\in\mathbb N\}$) only.

\section{Variance reduction via batch SGD}
\label{sec:generalization-sgd}

Paralleling the Euclidean setting, the function  $-(T_{\mu_k}^{m_k}-I)$  can be seen as an estimator of the gradient in the $k-$th step of the SGD scheme.  Hence,  conditionally on $m_0,\dots, m_{k-1}$, its integrated variance  given by $2F(\mu_k) -  \lVert  F^\prime(\mu_k) \rVert_{L^2(\mu_k)}^2$ could be large for some  sampled $\mu_k$, which would yield a slow convergence if the steps $\gamma_k$ are not small enough. To cope with this issue, as customary in stochastic algorithms, we are lead to consider alternative estimators of $F^\prime(\mu)$ with less (integrated) variance:

\begin{definition}
	Let $\mu_0 \in K_\Pi$, $m_k^i \stackrel{\text{iid}}{\sim}  \Pi$, and $\gamma_k > 0$ for $k \geq 0$ and $i=1,\dots,S_k$. The batch stochastic gradient descent (BSGD) sequence is given by 
	\begin{align}		\label{eq:sgd-batch-seq} \textstyle
	\textstyle\mu_{k+1} := \left[(1-\gamma_k)I + \gamma_k \frac{1}{S_k}\sum_{i=1}^{S_k} T_{\mu_k}^{m_k^i}\right](\mu_k).
	\end{align}
\end{definition}

Notice that $\frac{1}{S_k}\sum_{i=1}^{S_k} T_{\mu_k}^{m_k^i} - I$ is an unbiased estimator of $-F^\prime(\mu_k)$. {\color{black} Proceeding as in Proposition \ref{prop:bound-step}, with  $\mathcal F_{k+1}$ now denoting the sigma-algebra generated by $\{m_\ell^i:\, \ell\leq k,\,i\leq S_k\}$, and  writing $\Pi$ also for the law of an i.i.d.~sample of size $S_k$}, we now  have
\begin{align*}
\mathbb{E}\left[F(\mu_{k+1})|\mathcal F_k\right] =&\textstyle F(\mu_k)+\gamma_k\langle  F^\prime(\mu_k), \int \left( \frac{1}{S_k}\sum_{i=1}^{S_k} T_{\mu_k}^{m_k^i} - I \right) \,\Pi(dm_k^1\cdots dm_k^{S_k})\rangle_{L^2(\mu_k)} \\&+\textstyle  \frac{\gamma_k^2}{2}\int\lVert \frac{1}{S_k}\sum_{i=1}^{S_k} T_{\mu_k}^{m_k^i} - I \rVert^2_{L^2(\mu_k)}\Pi(dm_k^1\cdots dm_k^{S_k})\\
=&\textstyle F(\mu_k) - \gamma_k\lVert F^\prime(\mu_k) \rVert^2_{L^2(\mu_k)} + \frac{\gamma_k^2}{2}\int\lVert \frac{1}{S_k}\sum_{i=1}^{S_k} T_{\mu_k}^{m_k^i} - I \rVert^2_{L^2(\mu_k)}\Pi(dm_k^1\cdots dm_k^{S_k})\\
\leq&\textstyle F(\mu_k) - \gamma_k\lVert F^\prime(\mu_k) \rVert^2_{L^2(\mu_k)} + \frac{\gamma_k^2}{2} \frac{1}{S_k}\sum_{i=1}^{S_k}\int\lVert T_{\mu_k}^{m_k^i} - I \rVert^2_{L^2(\mu_k)}\Pi(dm_k^i)\\
=&\textstyle(1+\gamma_k^2)F(\mu_k) - \gamma_k\lVert F^\prime(\mu_k) \rVert^2_{L^2(\mu_k)}.
\end{align*}
The arguments in the proof of Theorem \ref{thm:sgd-convergence} can then be easily adapted to get :

\begin{theorem}
	\label{thm:sgd-batch-convergence}
	Under the assumptions of Theorem \ref{thm:sgd-convergence} the BSGD sequence $\{\mu_t\}_{t\geq0}$ in equation~\eqref{eq:sgd-batch-seq} converges a.s.\ to the 2-Wasserstein barycenter of $\Pi$.
\end{theorem}

The supporting idea for using mini-batches is \emph{reducing the noise} of the estimator of $F^\prime(\mu)$, namely: 
\begin{proposition}\label{prop batch noise reduction}
	The integrated variance of {\color{black} the mini batch estimator  of fixed batch size $S$ for $F^\prime(\mu)$, given  by} $-\frac{1}{S}\sum_{i=1}^{S}(T_{\mu}^{m_i}-I)$, decreases linearly in the sample size. More precisely, we have $$\textstyle\mathbb{V}[-\frac{1}{S}\sum_{i=1}^{S}(T_{\mu}^{m_i}-I)] = \textstyle \frac{1}{S} \mathbb{V}[-(T_{\mu}^{m_1}-I)] = \frac{1}{S}\left[2F(\mu) - \lVert  F^\prime(\mu) \rVert_{L^2(\mu)}^2\right].$$
	\begin{proof}
		
	{\color{black}	In a similar way as shown in eq.~\eqref{eq:VE}, the integrated variance of the mini-batch estimator  is} 
		\begin{align*}
		\textstyle\mathbb{V}\left[-\frac{1}{S}\sum_{i=1}^{S}(T_{\mu}^{m_i}-I)\right] =& \textstyle \mathbb{E}\left[ \left\lVert -\frac{1}{S}\sum_{i=1}^{S}(T_{\mu}^{m_i}-I) \right\rVert_{L^2(\mu)}^2 \right] -   \left\lVert\mathbb{E}\left[ -\frac{1}{S}\sum_{i=1}^{S}(T_{\mu}^{m_i}-I) \right]\right\rVert_{L^2(\mu)}^2 \\
		=&\textstyle \mathbb{E}\left[ \left\lVert -\frac{1}{S}\sum_{i=1}^{S}(T_{\mu}^{m_i}-I) \right\rVert_{L^2(\mu)}^2 \right] - \left\lVert  F^\prime(\mu)\right \rVert_{L^2(\mu)}^2\,.
		\end{align*}
		The first term  can be expanded as
		\begin{align*}
		\textstyle\left\lVert -\frac{1}{S}\sum_{i=1}^{S}(T_{\mu}^{m_i}-I) \right\rVert_{L^2(\mu)}^2 =& \textstyle \frac{1}{S^2}\langle \sum_{i=1}^{S}(T_{\mu}^{m_i}-I),\sum_{j=1}^{S}(T_{\mu}^{m_j}-I)\rangle_{L^2(\mu)}\\
		=&\textstyle \frac{1}{S^2}\sum_{i=1}^{S}\sum_{j=1}^{S}\langle T_{\mu}^{m_i}-I,T_{\mu}^{m_j}-I\rangle_{L^2(\mu)}\\
		=& \textstyle\frac{1}{S^2}\sum_{i=1}^{S}\lVert -(T_{\mu}^{m_i}-I) \rVert_{L^2(\mu)}^2 + \frac{1}{S^2}\sum_{j \neq i}^{S}\langle T_{\mu}^{m_i}-I,T_{\mu}^{m_j}-I\rangle_{L^2(\mu)}.
		\end{align*}
		By taking expectation, as the samples $m_i \sim \Pi$ are independent, we get
		\begin{align*}
		\textstyle\mathbb{E}\left[ \lVert -\frac{1}{S}\sum_{i=1}^{S}(T_{\mu}^{m_i}-I) \rVert_{L^2(\mu)}^2 \right]= &\textstyle \frac{1}{S^2}\sum_{i=1}^{S}\mathbb{E}\left[ W_2^2(\mu,m_i) \right] + \frac{1}{S^2}\sum_{j \neq i}^{S}\langle \mathbb{E}\left[ T_{\mu}^{m_i}-I\right], \mathbb{E}\left[T_{\mu}^{m_j}-I\right]\rangle_{L^2(\mu)}\\
		=&\textstyle \frac{2}{S^2}\sum_{i=1}^{S}F(\mu) + \frac{1}{S^2}\sum_{j \neq i}^{S}\langle F^\prime(\mu), F^\prime(\mu)\rangle_{L^2(\mu)} =\frac{2}{S}F(\mu) + \frac{S-1}{S}\lVert  F^\prime(\mu) \rVert_{L^2(\mu)}^2.	
		\end{align*}
  Subtracting $\lVert  F^\prime(\mu) \rVert_{L^2(\mu)}^2$ yields the asserted identities. 
	\end{proof}
\end{proposition}
The mini-batch implementation is easily seen to inherit convergence estimates established in Theorem \ref{thm:uKarcherSGDrate}:

\begin{theorem}\label{thm:uKarcherBSGDrate}
Under the assumptions of  Theorem \ref{thm:uKarcherSGDrate}, 
the BSGD sequence $\{\mu_k\}_k$ in eq.~\eqref{eq:sgd-batch-seq} is a.s.~convergent to $\hat{\mu}\in K_\Pi$ as soon as \eqref{eq:cond1-gamma} and \eqref{eq:cond2-gamma}  hold. Moreover, for every $a>C_{\hat{\mu}}^{-1}$ and $b\geq a$ there exists an explicit constant  $C_{a,b}>0$ such that, if $\gamma_k=\frac{a}{b+k}$ for all $k\in \NN$, the expected optimallity gap satisfies: 
$$ \textstyle \mathbb{E}\left[F(\mu_{k}) - F(\hat{\mu}) \right] \leq \frac{C_{a,b}}{b+k} \, .$$
\end{theorem}

\section{SGD for closed-form Wasserstein barycenters}
\label{sec special cases}

As discussed in the introduction, recent developments  enable  the  approximated computation of  optimal transport maps between two given absolutely-continuous distributions,  and hence the practical implementation of the SGD in fairly general cases is in principle feasible. We will not address  this issue in the present work, but rather content ourselves with analysing some families of models considered in \cite{cuestaalbertos1993optimal, alvarez2018wide}, for which this additional algorithmic aspect can be avoided,  their optimal transport maps being explicit and easy to evaluate.  Further, we will examine some of22  their closure properties which are preserved under the operation of \emph{taking barycenter}. This is important, for instance, in the context of the statistical application \cite{backhoff2018bayesian}, wherein $\Pi$ really represents a posterior distribution on models and its barycenter is postulated as a useful representative of the posterior distribution on models. It is hence desirable that the representative model shares some of the properties of all the models charged by the posterior. 

{ In the settings that will be presented, the pseudo-associativity condition \eqref{eq:pseudo-asoc_ineq} is either satisfied or conditions ensuring it can be given explicitly. Convergence bounds as in  Theorem \ref{thm:uKarcherSGDrate} can be easily deduced in those cases for each specification of ${\Pi}$ satisfying the required assumptions.}

\subsection{Univariate distributions}\label{Ex:1dim}
We assume that $m$ is a continuous distribution over $\mathbb{R}$, 
and denote respectively by $F_m$ and $Q_m:=F_m^{-1}$ its cumulative distribution function and its right-continuous quantile function. {\color{black} The increasing transport map from a continuous $m_0$ to $m$, also known as the monotone rearrangement, is given by $$T_{m_0}^m(x) = Q_{m}(F_{m_0}(x)),$$  and is known to be $p$-Wasserstein  optimal for $p\geq  1$ (see \cite[Remark 2.19(iv)]{villani2003topics}).} 
Given $\Pi$, the barycenter $\hat{m}$ is also independent of $p$ and characterized via its quantile, i.e.,
\begin{equation}\label{eq:quantcharac}
\textstyle Q_{\hat{m}}(\cdot) = \int Q_m(\cdot) \Pi(dm).
\end{equation}
{In words: the quantile function of the barycenter is equal to the average quantile function.}
Our SGD  iteration, starting from a distribution function $F_{\mu}(x)$, sampling some $m \sim \Pi$, and with step $\gamma$, produces the measure $$\nu =  ((1-\gamma)I + \gamma T_\mu^m)(\mu). $$ This is characterized by its quantile function $$Q_{\nu}(\cdot) = (1-\gamma)Q_{\mu}(\cdot) + \gamma Q_m(\cdot).$$ The batch stochastic gradient descent iteration equals $Q_{\nu}(\cdot) = (1-\gamma)Q_{\mu}(\cdot) + \frac{\gamma}{S} \sum_{i=1}^{S} Q_{m^{i}}(\cdot)$.

{As explained in Remark \ref{rem:assocKarcher}, the barycenter is automatically pseudo-associative, since transport maps (i.e.\  increasing functions) are in this case associative in the sense discussed therein, hence Theorem  \ref{thm:uKarcherSGDrate} applies. Moreover by Remark \ref{rem:VI+PLassocKarcher} it entails convergence bounds in $W_2^2$ for the SGD sequence itself. }

Interestingly, the model average $\bar{m}:=\int m\Pi(dm)$ is characterized by the \emph{averaged cumulative distribution function}, i.e., $F_{\bar{m}}(\cdot) = \int F_m(\cdot) \Pi(dm)$. The model average does not preserve intrinsic \emph{shape} properties from the distributions such as symmetry or unimodality. For example, if $\Pi = 0.3*\delta_{m_1} + 0.7*\delta_{m_2}$ with $m_1 = \mathcal{N}(1,1)$ and $m_2=\mathcal{N}(3,1)$, the \emph{model average} is an asymmetric bimodal distribution with modes on $1$ and $3$, while the Wasserstein barycenter is the Gaussian distribution {\color{black}$\hat m = \mathcal{N}(2.4,1)$. }

The following reasoning {\color{black} illustrates} the fact that Wasserstein barycenters preserve \emph{geometric properties} in a way that e.g.\ the model average does not. {\color{black}  A continuous distribution $m$ on $\mathbb{R}$ is unimodal with a mode on $\tilde{x}_m$ if its quantile function $Q(y)$ is concave for $y < \tilde{y}_m$ and convex for $y > \tilde{y}_m$, where $Q(\tilde{y}_m) = \tilde{x}_m$. Likewise, $m$ is symmetric w.r.t.  $x_m \in \mathbb{R}$ if $Q(\frac{1}{2}+y) = 2x_m - Q(\frac{1}{2}-y)$ for $y \in [0,\frac{1}{2}]$ (note that $x_m =Q_m(\frac{1}{2})$)}. These properties  can be analogously described in terms of the function $F_m$. Let us show that the barycenter preserves unimodality/symmetry:


\begin{proposition}\label{prop:symmetric-barycenter}
	If $\Pi$ is concentrated on continuous symmetric (resp.\ symmetric unimodal) univariate distributions, then the barycenter $\hat{m}$ is symmetric (resp.\ symmetric unimodal).
\end{proposition}
	\begin{proof}
		{\color{black} Using the quantile function characterization \eqref{eq:quantcharac}, we have for  $y \in [0,\frac{1}{2}]$ that 
		\begin{align*}
		\textstyle Q_{\hat{m}}\left(\frac{1}{2}+y\right) &= \textstyle \int Q_m\left(\frac{1}{2}+y\right) \Pi(dm) 
		= \textstyle 2x_{\hat{m}}-Q_{\hat{m}}\left(\frac{1}{2}-y\right),
		\end{align*} 
		where $x_{\hat{m}}: = \int x_m \Pi(dm)$. In other words $\hat{m}$ is symmetric w.r.t. $x_{\hat{m}}$. Now, if each $m$ is  unimodal in addition to symmetric, its mode $\tilde{x}_m$ coincides with its median $Q_m(\frac{1}{2})$, and $Q_m(\lambda y + (1-\lambda)z )\geq (\mbox{resp.}  \leq )  \lambda Q_m( y) + (1-\lambda)Q_m(z) $ for all  $y,z < (\mbox{resp.} >) \frac{1}{2}$ and $\lambda \in [0,1]$. Integrating these inequalities w.r.t. $\Pi(dm)$, and using  \eqref{eq:quantcharac} we deduce that $\hat{m}$ is unimodal (with mode $1/2$) too. }
	\end{proof}
Unimodality is not preserved in general non-symmetric cases. Still, for some families of distributions unimodality is maintained after taking barycenter, as we show in the next:
\begin{proposition}\label{prop:log-concave-barycenter}
	If $\Pi\in\mathcal W_p(\mathcal W_{ac}(\mathbb R))$ is concentrated on log-concave univariate distributions, then the barycenter $\hat{m}$ is unimodal.
\end{proposition}
	\begin{proof}
		If $f(x)$ is a log-concave density, then $-\log(f(x))$ is convex so is $\exp(-\log(f(x)) = \frac{1}{f(x)}$. Some computations reveal that $f(x)dx$ is unimodal for some $\tilde{x} \in \mathbb{R}$, so its quantile $Q(y)$ is concave for $y < \tilde{y}$ and convex for $y > \tilde{y}$ where $Q(\tilde{y}) = \tilde{x}$. Since $\frac{1}{f(x)}$ is convex decreasing for $x < \tilde{x}$ and convex increasing for $x > \tilde{x}$, then $\frac{1}{f(Q(y))}$ is convex. So $\frac{dQ}{dy}(y) = \frac{1}{f(Q(y))}$ is {\color{black} convex, positive and with minimum at $\tilde{y}$}. Given $\Pi$, its barycenter $\hat{m}$ fulfills $$\textstyle \frac{dQ_{\hat{m}}}{dy} = \int \frac{dQ_m}{dy} \Pi(dm),$$ so if all $\frac{dQ_m}{dy}$ are convex, then $\frac{dQ_{\hat m}}{dy}$ is {\color{black} convex and positive,  with a minimum at some $\hat{y}$. Thus,  $Q_{\hat{m}}(y)$ is concave for $y < \hat{y}$ and convex for $y > \hat{y}$ and $\hat{m}$ is unimodal with  mode at $ Q_{\hat{m}}(\hat{y})$.}
	\end{proof}

{\color{black} Useful examples of log-concave distribution families include the general normal,  exponential, logistic, Gumbel, chi-square and Laplace laws, as well as the  Weibull, power, gamma and beta families when their shape parameters are larger than one. }

\subsection{Distributions sharing a common copula}
\label{sec:sharing-copula}
If two multivariate distributions $P$ and $Q$ over $\mathbb{R}^q$ share the same copula, then their $W_p$ distance to the $p$-th power is the sum of the $W_p(\mathbb R)$ distances between their marginals raised to the $p$-power. Furthermore, if the marginals of $P$ have no atoms, then an optimal map is given by the coordinate-wise transformation 
$T(x) = (T^1(x_1),\dots,T^q(x_q))$ where $T^i(x_i)$ is the monotone rearrangement between the marginals $P^i$ and $Q^i$ for $i=1,\dots,q$. 
This setting allows us to easily extend the results from the univariate case to the multidimensional case.
\begin{lemma}\label{share-copula-barycenter}
	If $\Pi\in\mathcal W_p(\mathcal W_{ac}(\mathbb R^q))$ is concentrated on a set of measures sharing the same copula $C$, then the $p$-Wasserstein barycenter $\hat{m}$ of $\Pi$ has copula $C$ as well, and its $i$-th marginal $\hat{m}^i$ is the barycenter of the $i$-th marginal measures of $\Pi$. In particular the barycenter does not depend on $p$.
	\end{lemma}
\begin{proof}
		It is known (\cite{cuestaalbertos1993optimal,alfonsi2014remark}) that for two distributions $m$ and $\mu$ with i-th marginals $m^i$ and $\mu^i$ for $i=1,...,q$ respectively, the $p$-Wasserstein metric satisfies $$\textstyle W_p^p(m,\mu) \geq \sum_{i=1}^n W_p^p(m^i,\mu^i),$$ where equality is reached if $m$ and $\mu$ share the same copula $C$ (we abused notation denoting $W_p$ the $p$-Wasserstein distance on $\mathbb R^q$ as well as on $\mathbb R$). Thus,
		\begin{align*}
		\textstyle\int W_p^p(m,\mu)\Pi(dm) &\textstyle\geq \int \sum_{i=1}^q W_p^p(m^i,\mu^i)\Pi(dm)=  \sum_{i=1}^q \int W_p^p(\nu,\mu^i)\Pi^i(d\nu),
		\end{align*}
		where $\Pi^i$ is defined via the identity $\int_{\mathcal P(\mathbb R)} f(\nu)\Pi^i(d\nu)= \int_{\mathcal P(\mathbb R^q)}  f(m^i)\Pi(dm)$.
		The infimum for the lower bound is reached on the univariate measures $\hat{m}^1,...,\hat{m}^q$ where $\hat{m}^i$ is the $p$-barycenter of $\Pi^i$, which means that $\hat{m}^i = \argmin \int W_p^p(\nu,\mu^i)\Pi^i(d\nu)$. It is plain that the infimum is reached on the distribution $\hat{m}$ with copula $C$ and $i$-th marginal $\hat{m}^i$ for $i=1,...,q$, which then has to be the barycenter of $\Pi$ and is independent of $p$.
	\end{proof}

A Wasserstein SGD iteration, starting from a distribution $\mu$, sampling $m \sim \Pi$, and with step $\gamma$, both $\mu$ and $m$ having copula $C$, produces the measure $\nu = ((1-\gamma)I + \gamma T_\mu^m)(\mu)$ characterized by having copula C and the $i$-th marginal quantile functions $$Q_{\nu^i}(\cdot) = (1-\gamma)Q_{\mu^i}(\cdot) + \gamma Q_{m^i}(\cdot),$$ for $i=1,\dots,q$. The batch stochastic gradient descent iteration works analogously. {Alternatively, one can perform (batch) stochastic gradient descent component-wise (with respect to the marginals $\Pi^i$ of $\Pi$) and then make use of the copula $C$.}
{As in the one-dimensional case, the barycenter is in this case automatically pseudo-associative since it is associative. Bounds for the expected optimality gap and in $W_2^2$ for the SGD sequence can be similarly deduced in this case too.  }

\subsection{Spherically equivalent distributions}\label{Ex:sphere}
We denote here by $\mathcal{L}(\cdot)$ the law of a random vector, so $m=\mathcal L(x)$ and $x\sim m$ are synonyms.
Following \cite{cuestaalbertos1993optimal}, another multidimensional case is constructed as follows: Given a fixed measure $\tilde{m} \in \mathcal{W}_{2,ac}(\mathbb{R}^q)$, its associated family of spherically equivalent distributions is $$\textstyle \mathcal{S}_0:=\mathcal{S}(\tilde{m}) = \left\lbrace\mathcal{L}\left(\frac{ \alpha(\lVert\tilde{x}\rVert_2)}{\lVert \tilde{x}\rVert_2}\tilde{x}\right)| \alpha \in \mathcal{ND}(\mathbb{R}), \tilde{x} \sim \tilde{m} \right\rbrace,$$
where $\lVert \, \rVert_2$ is the Euclidean norm and $\mathcal{ND}(\mathbb{R})$ is the set of non-decreasing non-negative functions of $\mathbb R_+$. These type of distributions include the simplicially contoured distributions, and also elliptical distributions with the same correlation structure. 

If $y \sim m \in \mathcal{S}_0$, then  $\alpha(r) = Q_{\lVert y \rVert_2}(F_{\lVert \tilde{x} \rVert_2}(r))$, where $Q_{\lVert y \rVert_2}$ is the quantile function of the norm of $y$, $F_{\lVert \tilde{x} \rVert_2}$ is the distribution function of the norm of $\tilde{x}$, and $y\sim \frac{ \alpha(\lVert\tilde{x}\rVert_2)}{\lVert \tilde{x}\rVert_2}\tilde{x} $. More generally, if $m_1=\mathcal{L}\left(\frac{ \alpha_1(\lVert\tilde{x}\rVert_2)}{\lVert \tilde{x}\rVert_2}\tilde{x}\right)$ and $m_2=\mathcal{L}\left(\frac{ \alpha_2(\lVert\tilde{x}\rVert_2)}{\lVert \tilde{x}\rVert_2}\tilde{x}\right)$, then the optimal transport from $m_1$ to $m_2$ is given by $T_{m_1}^{m_2}(x)=\frac{ \alpha(\lVert x\rVert_2)}{\lVert x\rVert_2}x$ where $\alpha(r) =Q_{\lVert x_2 \rVert_2}(F_{\lVert x_1 \rVert_2}(r))$.  Since $F_{\lVert x_1 \rVert_2}(r)=F_{\lVert \tilde{x} \rVert_2}(\alpha_1^{-1}(r))$ and $Q_{\lVert x_2 \rVert_2}(r)=\alpha_2(Q_{\lVert \tilde{x} \rVert_2}(r))$, we see that $\alpha(r) =\alpha_2(Q_{\lVert \tilde{x} \rVert_2}(F_{\lVert \tilde{x} \rVert_2}(\alpha_1^{-1}(r)))) = \alpha_2(\alpha_1^{-1}(r))$, so finally $$\textstyle T_{m_1}^{m_2}(x)=\frac{ \alpha_2(\alpha_1^{-1}(\lVert x\rVert_2))}{\lVert x\rVert_2}x.$$
Note that these kind of transports are closed under composition, convex combination, and contain the identity. A stochastic gradient descent iteration, starting from a distribution $\mu = \mathcal{L}\left(\frac{ \alpha_0(\lVert\tilde{x}\rVert_2)}{\lVert \tilde{x}\rVert_2}\tilde{x}\right)$, sampling $m=\mathcal{L}\left(\frac{ \alpha(\lVert\tilde{x}\rVert_2)}{\lVert \tilde{x}\rVert_2}\tilde{x}\right) \sim \Pi$, with step $\gamma$, produces $m_1 =T_0^{\gamma,m}(\mu):=  ((1-\gamma)I + \gamma T_\mu^m)(\mu)$. Since $T_0^{\gamma,m}(x) = \frac{ (\gamma\alpha + (1-\gamma)\alpha_0)(\alpha_0^{-1}(\lVert x\rVert_2))}{\lVert x\rVert_2}x$, we have that $m_1 = \mathcal{L}\left(\frac{ \alpha_1(\lVert\tilde{x}\rVert_2)}{\lVert \tilde{x}\rVert_2}\tilde{x}\right)$ with $\alpha_1 = \gamma\alpha + (1-\gamma)\alpha_0$. Analogously, the batch stochastic gradient iteration produces $$\textstyle \alpha_{1} = (1-\gamma)\alpha_0 + \frac{\gamma}{S}\sum_{i=1}^{S} \alpha_{m^{i}}.$$ Note that these iterations live in $\mathcal{S}_0$, thus, so does the barycenter $\hat{m} \in \mathcal{S}_0$.

For the barycenter $\hat{m}=\mathcal{L}\left(\frac{ \hat{\alpha}(\lVert\tilde{x}\rVert_2)}{\lVert \tilde{x}\rVert_2}\tilde{x}\right)$, the equation $\int T_{\hat{m}}^{m}(x)\Pi(dm) = x$ can be expressed as $ \hat{\alpha}(r) = \int \alpha_m(r)\Pi(dm)$, or equiv.\  $Q^{\hat{m}}_{\lVert \hat{y} \rVert_2}(p) = \int Q^m_{\lVert y \rVert_2}(p)\Pi(dm)$, where $Q^m_{\lVert y \rVert_2}$ is the quantile function of the norm of $y \sim m$. This is similar to the univariate setting {and, as in that case, the barycenter is associative, and convergence bounds for the optimality gap and the SDG sequence can be established}.

\subsection{Scatter-location family}\label{sec scat loc}
We borrow here the setting of \cite{alvarez2018wide}, where another useful multidimensional case is defined as follows: Given a fixed distribution $\tilde{m} \in \mathcal{W}_{2,ac}(\mathbb{R}^q)$, referred to as \emph{generator}, the generated scatter-location family is given by $$\mathcal{F}(\tilde{m}) : = \{\mathcal{L}(A\tilde{x}+b)| A \in \mathcal{M}_+^{q\times q} ,b \in \mathbb{R}^q, \tilde{x} \sim \tilde{m} \},$$
where $\mathcal{M}_+^{q\times q}$ is the set of symmetric positive definite matrices of size $q\times q$. Without loss of generality we can assume that $\tilde{m}$ has zero mean and identity covariance. {\color{black} A clear example is the  multivariate Gaussian family  $\mathcal{F}(\tilde{m})$,   with $\tilde{m}$ a  standard multivariate normal distribution.} 

The optimal transport between two members of $\mathcal{F}(\tilde{m})$ is explicit. If $m_1 = \mathcal{L}(A_1\tilde{x}+b_1)$ and $m_2 = \mathcal{L}(A_2\tilde{x}+b_2)$ then the optimal map from $m_1$ to $m_2$ is given by $T_{m_1}^{m_2}(x) = A(x-b_1) + b_2$ where $A = A_1^{-1}(A_1A_2^2A_1)^{1/2}A_1^{-1} \in \mathcal{M}_+^{q\times q}$.  {
The family $\mathcal{F}(\tilde{m})$ is therefore geodesically closed.  }

If $\Pi$ is supported on $\mathcal{F}(\tilde{m})$, then its $2$-Wasserstein barycenter $\hat{m}$ belongs to $\mathcal{F}(\tilde{m})$. Call its mean $\hat{b}$ and its covariance matrix $\hat{\Sigma}$. Since the optimal map from $\hat{m}$ to $m$ is $T_{\hat{m}}^{m}(x) = A_{\hat{m}}^{m}(x-\hat{b}) + b_m$ where $A_{\hat{m}}^{m} = \hat{\Sigma}^{-1/2}(\hat{\Sigma}^{1/2}\Sigma_m\hat{\Sigma}^{1/2})^{1/2}\hat{\Sigma}^{-1/2}$ and we know that $\hat{m}$-almost surely $\int T_{\hat{m}}^{m}(x)\Pi(dm) = x$, then we must have that $\int A_{\hat{m}}^{m} \Pi(dm) = I$, since clearly $\hat{b}=\int b_m \Pi(dm)$. As a consequence, we have $\hat{\Sigma} = \int (\hat{\Sigma}^{1/2}\Sigma_m\hat{\Sigma}^{1/2})^{1/2}\Pi(dm)$. 

A stochastic gradient descent iteration, starting from a distribution $\mu = \mathcal{L}(A_0\tilde{x}+b_0)$, sampling some $m=\mathcal{L}(A_m\tilde{x}+b_m) \sim \Pi$, and with step $\gamma$, produces the measure $\nu=T_0^{\gamma,m}(\mu):= ((1-\gamma)I + \gamma T_\mu^m)(\mu)$. If $\tilde{x}$ has a multivariate distribution $\tilde{F}(x)$, then $\mu$ has distribution $F_0(x)=\tilde{F}(A_0^{-1}(x-b_0))$ with mean $b_0$ and covariance $\Sigma_0 = A_0^2$. We have that $T_0^{\gamma,m}(x) =  ((1-\gamma)I + \gamma A_{\mu}^{m})(x-b_0) + \gamma b_m + (1-\gamma)b_0$ with $A_{\mu}^{m} := A_0^{-1}(A_0A_m^2A_0)^{1/2}A_0^{-1}$. Then  $$F_{\nu}(x)=:F_1(x) = F_0([T_0^{\gamma.m}]^{-1}(x)) =\tilde{F}( [(1-\gamma)A_0+\gamma A_{\mu}^{m}A_0]^{-1}(x-\gamma b_m - (1-\gamma)b_0 ) ),$$ with mean $b_1 = (1-\gamma)b_0 + \gamma b_m$ and covariance
\begin{align*}
\Sigma_1 = A_1^2&= [(1-\gamma)A_0+\gamma A_0^{-1} (A_0A_m^2A_0)^{1/2}][(1-\gamma)A_0+\gamma (A_0A_m^2A_0)^{1/2}A_0^{-1}]\\
&=  A_0^{-1}[(1-\gamma)A_0^2+\gamma  (A_0A_m^2A_0)^{1/2}][(1-\gamma)A_0^2+\gamma (A_0A_m^2A_0)^{1/2}]A_0^{-1}\\
&=A_0^{-1}[(1-\gamma)A_0^2+\gamma (A_0A_m^2A_0)^{1/2}]^2A_0^{-1}
\end{align*}
The batch stochastic gradient descent iteration is characterized by
\begin{align*}
b_{1} = \textstyle (1-\gamma)b_0+\frac{\gamma}{S}\sum_{i=1}^{S}b_{m^{i}}\,\text{and}\, 
A_{1}^2 = A_0^{-1}[(1-\gamma)A_0^2+\frac{\gamma}{S}\sum_{i=1}^{S} (A_0A_{m^{i}}^2A_0)^{1/2}]^2A_0^{-1}.
\end{align*}

{ Notice that, if all the matrices $\{A_m:m\in\text{supp}(\Pi)\}$ share the same eigenspaces, then the barycenter is pseudo associative as it is in fact associative. In case these matrices do not necessarily share the same eigenspaces, and $\tilde m$ is Gaussian, it is still possible to give conditions under which the barycenter is pseudo associative: as we have already mentioned, this property  holds  in the setting of \cite{chewi2020gradient} under conditions  of uniform boundedness and uniform positivity of covariance matrices.  }
\medskip

{\bf Acknowledgments }
We thank two anonymous referees for their valuable comments and questions that motivated some improvements of our results and of the presentation of the paper. This work was partially funded by the  ANID-Chile grants: Fondecyt-Regular 1201948 (JF) and 1210606 (FT); Center for Mathematical Modeling ACE210010 and FB210005 (JF, FT); and Advanced Center for Electrical and Electronic Engineering FB0008 (FT). 

\appendix
\section{Some remarks in the case of general measures}

We discuss a natural counterpart to the SGD sequence in the case where $\Pi$ is not necessarily supported on absolutely continuous measures. {However in this case we neither have a verifiable result guaranteeing the convergence of the SGD method, nor a verifiable result saying that the accumulation points of the SGD sequence are Karcher means necessarilly. The goal of this section is to introduce the objects needed for an analysis of the general case, to advance the convergence analysis as much as possible, and to illustrate the difficulties that we encounter which make us stop short of obtaining general convergence results.}

While we retain Assumption (A1), we modify Assumption (A2) into
\begin{enumerate}
\item[(A2')] $\Pi = \sum_{j=1}^\ell \lambda_j\delta_{\nu_j}$  with each $\nu_j\in \mathcal W_2(\mathcal X)$ and where the $\lambda_j\in (0,1)$ sum to 1. 
\end{enumerate}

\begin{remark}\label{rem:Kpi_non_abs}
   We can find $V:\mathcal X\to [0,\infty)$ convex, continuous, and super-quadratic (i.e.\ $\lim_{|y|\to\infty} V(y)/|y|^2=+\infty$) and $C\in(0,\infty)$ s.t.\ $\int Vd\nu_j\leq C$ for each $j$. Defining 
   \begin{align}\label{eq:appen_KPi}
       \textstyle K_\Pi=\left\{\mu:\int Vd\mu\leq C\right\},
   \end{align} 
   it follows that $K_\Pi$ is compact in $\mathcal W_2(\mathcal X)$ and $\Pi(K_\Pi)=1$. Moreover, if $(X,Y)$ is any optimal coupling with marginals $\mu$ and $\nu$, with $\mu$ and $\nu$ in $K_\Pi$, then also $\text{Law}(tX+(1-t)Y)\in K_\Pi$ for each $t\in(0,1)$. {In the sequel we denote by $K_\Pi$ a set with these properties which may or may not be given explicitly as in \eqref{eq:appen_KPi}.}
\end{remark}

Throughout we denote by $\Opt(\mu,\nu)$ the set of optimal couplings attaining the infimum defining $W_2(\mu,\nu)$, which is non-empty and may contain multiple elements, and we fix
\begin{align}
    \mathcal W_2(\mathcal X)^2\ni(\mu,\nu)\mapsto \opt(\mu,\nu)\in \Opt(\mu,\nu),
\end{align}
a measurable selection of optimal couplings. In practical terms an algorithm for computing or approximating optimal couplings would be automatically measurable.

\begin{definition}
	Let {\color{black}$\mu_0 \in K_\Pi$}, $m_k \stackrel{\text{iid}}{\sim}  \Pi$,  and $\gamma_k > 0$ for $k \geq 0$. We define the stochastic gradient descent (SGD) sequence by
	\begin{align}
	\label{eq:sgd-seq_extended}
	\mu_{k+1} := \text{Law} \left[(1-\gamma_k)X + \gamma_k Y\right],
	\end{align}
 where $\text{Law} (X,Y)=G(\mu_k,m_k)$.
\end{definition}
By Remark \ref{rem:Kpi_non_abs} we have $\{\mu_k:k\in\mathbb N\}\subset K_\Pi$ (a.s.).
\begin{definition}\label{def:Karch_app}
    $\mu\in \mathcal W_2(\mathcal X)$ is a Karcher mean of $\Pi$ if
    $$\textstyle (\mu(dx)-a.s.)\,\, x=\sum_{j=1}^\ell \lambda_j \mathbb E[Y_j|X=x],$$
    where (for each $j$) $\text{Law} (X,Y_j)$ is some coupling in $\Opt (\mu,\nu_j)$ .
\end{definition}

We first observe that if $\mu$ is absolutely continuous, then it is a Karcher mean according to Definition \ref{def:Karch_app} iff its is a Karcher mean according to Definition \ref{defi:Karcher_intro}. On the other hand, if $\mu$ is a barycenter of $\Pi$ then it is also a Karcher mean according to Definition \ref{def:Karch_app}. Indeed, if $\text{Law} (X,Y_j)\in\Opt (\mu,\nu_j)$ and we couple all these random variables in the same probability space, then
\[\textstyle \sum_{i=1}^\ell \lambda_i W_2^2(\mu,\nu_i) = \mathbb E\left[\sum_{i=1}^\ell  \lambda_i |X-Y_i|^2\right] \geq  \mathbb E\left[\sum_{j=1}^\ell\lambda_j\left|Y_j-\sum_{i=1}^\ell  \lambda_i Y_i \right|^2\right] \geq \sum_{i=1}^\ell \lambda_i W_2^2(\tilde \mu,\nu_i) ,\]
so the inequalities above must be actual equalities and we have $\mu=\text{Law}(\sum_i  \lambda_i Y_i):=\tilde \mu$ as well as $X= \sum_i  \lambda_i Y_i$ (a.s.). Taking conditional expectation w.r.t.\ $X$ in the latter, we conclude.

In direct analogy to the absolutely continuous case, we denote $$\textstyle F(\mu):=\frac{1}{2}\sum_j \lambda_j W_2^2(\mu,\nu_j),$$ and we introduce 
$$\textstyle F'(\mu)(x):= x - \sum_{j=1}^\ell\lambda_j \mathbb E[Y_j|X=x], $$
where $(X,Y_j)\sim G(\mu,\nu_j)$.  With this notation we have the implication:  $\| F'(\mu)\|_{L^2(\mu)}=0 \,\,\Rightarrow \,\, \mu $ is a Karcher mean. As before we denote by $\mathcal F_{0}$ the trivial sigma-algebra and $\mathcal F_{k+1},k\geq 0,$  the sigma-algebra generated by $m_0,\dots,m_{k}$.
\begin{proposition}
	\label{prop:bound-step_gnral}
	The SGD sequence in eq.~\eqref{eq:sgd-seq_extended} verifies (a.s.) that 
	\begin{align}
	\label{eq:sgd-ineq_apendix}
	\mathbb{E}\left[F(\mu_{k+1}) - F(\mu_k)|\mathcal F_{k}\right] \leq \gamma_k^2F(\mu_k) - \gamma_k\lVert F^\prime(\mu_k)\rVert^2_{L^2(\mu_k)}.
	\end{align}
	\begin{proof} 

 We can define a coupling $(X_k,Y_k,Z_i)$ such that $\text{Law}(X_k,Y_k)\in G(\mu_k,m_k)$, $\text{Law}(X_k,Z_i)\in G(\mu_k,\nu_i)$, and where $Y_k$ and $Z_i$ are independent conditionally on $\mathcal F_k$. Then the coupling $((1-\gamma_k)X_k+\gamma_k Y_k,Z_i)$ has first marginal $\mu_{k+1}$ and second marginal $\nu_i$. We compute
 \begin{align*}
  W_2^2(\mu_{k+1},\nu_i) 	  	\leq \mathbb E[|(1-\gamma_k)X_k+\gamma_k Y_k-Z_i|^2] &= \mathbb E[|X_k-Z_i|^2] -2\gamma_k \mathbb E[ \langle X_k-Z_i, X_k-Y_k\rangle ] + \gamma_k^2\mathbb E[|X_k-Y_k|^2].
 \end{align*}
 Hence
 \begin{align*}
     F(\mu_{k+1}) =&  \textstyle \frac{1}{2}\int W_{2}^2(\mu_{k+1},\nu)\Pi(d\nu)\\
		\leq& \textstyle \frac{1}{2}\sum_{i=1}^\ell \lambda_i \mathbb E[|X_k-Z_i|^2]  - \gamma_k \mathbb E[ \langle X_k- \sum_{i=1}^\ell \lambda_i Z_i,X_k-Y_k\rangle ] +  \frac{1}{2}\gamma_k^2\mathbb E[|X_k-Y_k|^2]\\
  =& \textstyle F(\mu_k) -\gamma_k \mathbb E[ \langle X_k- \sum_{i=1}^\ell \lambda_i Z_i ,X_k-Y_k\rangle ]  + \frac{1}{2} \gamma_k^2 W_2^2(\mu_k,m_k).
 \end{align*}
 Taking conditional expectation with respect to $\mathcal F_k$, and as $m_k$ is independently sampled from this sigma-algebra, we see that  $$\textstyle \mathbb E[W_2^2(\mu_k,m_k)|\mathcal F_k]= \sum_{i=1}^\ell\lambda_i W_2^2(\mu_k,\nu_i)=2F(\mu_k).$$
 Similarly, we have
$$\textstyle \mathbb E[ \mathbb E[ \langle X_k- \sum_{i=1}^\ell \lambda_i Z_i,X_k-Y_k\rangle ] |\mathcal F_k]=\mathbb E[ \langle X_k- \sum_{i=1}^\ell \lambda_i\mathbb E[Z_i|X_k] ,X_k-\sum_{i=1}^\ell  \lambda_i \mathbb E[Y_i|X_k]\rangle ].
$$
We conclude that 
\begin{align*}
		\mathbb{E}\left[F(\mu_{k+1})|\mathcal F_k\right] \leq &  \textstyle   (1+\gamma_k^2)F(\mu_k) -\gamma_k \lVert F^\prime(\mu_k)\rVert^2_{L^2(\mu_k)}. 
\end{align*}
	\end{proof}
\end{proposition}

{he next lemma is the only part where we use that $\Pi$ is supported in finitely many measures. With more refined arguments one could surely avoid this limitation.}

\begin{lemma}
	\label{continuity norm F'v Appendix}
	 Let $(\rho_n)_n \subset \mathcal{W}_{2}(\RR^q)   $  be a sequence converging w.r.t.\  $W_2$ to $\rho\in  \mathcal{W}_{2}(\RR^q) $. Then, as $n\to \infty$,we have \textcolor{black}{ i) :  $F(\rho_n)\to  F(\rho)$ and ii) : 
	  $\| F'(\rho_n)\|_{L^2(\rho_n)}\to 0$ implies that $\rho$ is a Karcher mean.}
\end{lemma}

\begin{proof}
    Part (i) is immediate since $W_2(\rho_n,\nu_j)\to W_2(\rho,\nu_j)$ for each $j$. For part (ii), we first observe that $G(\rho_n,\nu_j)$ is tight, for each $j$. Passing to a subsequence if necessary, Skorokhod representation theorem yields, on some probability space, a sequence of random variables $X^n\sim\rho_n$ and $Y_j^n\sim \nu_j$, as well as $X\sim \rho$ and $Y_j\sim\nu_j$, such that $X^n\to X$ and for each $j$ also $Y_j^n\to Y_j$ (convergence is a.s.\ and in $L^2$). The sequence $\{\mathbb E[Y_j^n|X^n]:n\in \mathbb N\}$ is $L^2$-bounded, and hence by repeatedly applying Komlos theorem, we find yet another subsequence (relabeled so it is indexed by $\mathbb N$) such that
    $$\textstyle \frac{1}{n}\sum_{r\leq n}\mathbb E[Y_j^r|X^r]\to Z_j,$$
    almost surely and in $L^2$, for each $j$. 
    
    We check that $\mathbb E[Z_j|X]=\mathbb E[Y_j|X] $. Indeed, if $g$ is bounded and continuous
    \begin{align*}
        \mathbb E[g(X)Z_j]&= \textstyle \lim_n \mathbb E\left[g(X)\left(\frac{1}{n}\sum_{r\leq n}E[Y_j^r|X^r]\right)\right]\\ &= \textstyle \lim_n
      \frac{1}{n}\sum_{r\leq n}  \mathbb E[g(X^r)E[Y_j^r|X^r]]+ \frac{1}{n}\sum_{r\leq n}\mathbb E[\{g(X)-g(X^r)\}\mathbb E[Y_j^r|X^r]] \\ &= \textstyle 
      \lim_n
      \frac{1}{n}\sum_{r\leq n}  \mathbb E[g(X^r)Y_j^r] = \mathbb E[g(X)Y_j],
    \end{align*}
    by dominated convergence. Also
    \begin{align*}
        0 &= \textstyle \lim_n \| F'(\rho_n)\|_{L^2(\rho_n)}^2 = \lim_n \mathbb E\left[\left | X^n-\sum_{j=1}^\ell\lambda_j \mathbb E[Y_j^n|X^n]  \right|^2\right] \\
        &= \textstyle \lim_n \frac{1}{n}\sum_{r\leq n}\mathbb E\left[\left| X^r-\sum_{j=1}^\ell\lambda_j \mathbb E[Y_j^r|X^r]  \right |^2\right] 
        \\ &\geq  \textstyle \lim_n \mathbb E\left[\left | \frac{1}{n}\sum_{r\leq n}X^r-\sum_{j=1}^\ell\lambda_j \frac{1}{n}\sum_{r\leq n} \mathbb E[Y_j^r|X^r]  \right |^2\right]  \\
        &\geq \textstyle \mathbb E\left[\left |X- \sum_{j=1}^\ell\lambda_j Z_j\right |^2\right] \\
        &\geq \textstyle  \mathbb E\left[\left |X- \sum_{j=1}^\ell\lambda_j \mathbb E[Z_j|X]\right |^2\right], 
    \end{align*}
    and we conclude that $X= \sum_{j=1}^\ell\lambda_j \mathbb E[Y_j|X]$ (a.s.). Finally we remark that $\text{Law}(X,Y_j)\in\Opt(\rho,\nu_j)$, and hence we conclude that $\rho$ is a Karcher mean.
\end{proof}


{We can now provide a convergence result. The hypotheses of this result are implied by the hypotheses of Theorem \ref{thm:sgd-convergence} in case that $\Pi$ is concentrated on absolutely continuous measures. However we stress that we do not know of any example where Theorem \ref{thm:sgd-convergence_appendix} is applicable and $\Pi$ is concentrated on non-absolutely continuous measures. For this reason we rather think that this result and its proof are a template of what \emph{could happen} or \emph{could be done} in the general case, and use it really to illustrate the difficulties present in the general case.}

\begin{theorem}
	\label{thm:sgd-convergence_appendix}
	Assume (A1), (A2'), conditions \eqref{eq:cond1-gamma} and \eqref{eq:cond2-gamma}, that $\Pi$ admits a unique Karcher mean (in the sense of Definition \ref{def:Karch_app}) in $K_\Pi$, {and that $K_\Pi$ contains a (hence, exactly one) 2-Wasserstein barycenter $\hat\mu$}. Then,  the SGD sequence $\{\mu_k\}_k$ in eq.~\eqref{eq:sgd-seq_extended} is a.s.~convergent to $\hat{\mu}$. 
\end{theorem}

\begin{proof}
The proof of Theorem \ref{thm:sgd-convergence} can be followed verbatim up to Equation \eqref{eq:liminf0}, namely
	\begin{equation*} \textstyle
	\liminf_{t\to \infty} \lVert F^\prime(\mu_t)\rVert^2_{L^2(\mu_t)}=0,\,\, \text{a.s.}
\end{equation*}

We observe that if $\rho_n\to\rho$ with $\{\rho_n\}_n\subset K_\Pi$ is s.t. $F(\rho_n )\geq F(\hat \mu)+\epsilon$, for $\epsilon >0$, and $\lVert F^\prime(\rho_n)\rVert^2_{L^2(\rho_n)}\to 0$, then by Lemma \ref{continuity norm F'v Appendix} $\rho$ is a Karcher mean in $K_\Pi$ which is necessarily different from $\hat \mu$. This would contradict the uniqueness of Karcher means in $K_\Pi$. Thus we establish
\begin{equation*}\textstyle 
	\forall \varepsilon>0 , \, \inf_{\{ \rho : F(\rho) \geq \hat F +\varepsilon \} \cap K_\Pi}  \,  \lVert F^\prime(\rho)\rVert^2_{L^2(\rho)}>0.  
	\end{equation*}
  From here on we can follow again the proof of Theorem \ref{thm:sgd-convergence}. 
\end{proof}

In the absolutely continuous case, covered in the rest of this paper, we required the Karcher mean to be absolutely continuous. Moreover in this case there is a unique barycenter (automatically absolutely continuous). Once we leave the absolutely continuous world we have to be more careful, as illustrated by the following example:

\begin{example}
Let $Z$ be a $\mathcal X$-valued random variable with finite second moment, and let $b_j\in\mathcal X$ with $\sum_j \lambda_j b_j =0$. Define $\nu_j:=Law(Z+b_j)$. Then $X\sim\mu$ is a Karcher mean iff $X=\mathbb E[Z|X]$ for some optimal coupling of $X$ and $Z$. Hence $X:=Z$ and $X:=\mathbb E[Z]$ are both Karcher means per Definition \ref{def:Karch_app}, even if $Z$ (and so each $\nu_j$) is absolutely continuous. On the other hand one can check that the barycenter of the $\nu_j$'s is unique and given by $Law(Z)$. 
\end{example}

\begin{example}
We continue in the the setting of the previous example and choose $\mathcal X=\mathbb R^2$. Here $Z=(B,0) $ with $B\in\{-1,1\}$ with equal probabilities, and $\nu_1=Law((B,1))$, $\nu_2=Law((B,-1))$. In this case we check that $\mu^p:=p(\delta_{(1,0)}+\delta_{(-1,0)})+(1-2p)\delta_{(0,0)}$ is a Karcher mean for each $p\in[0,1/2]$. In particular the barycenter, given by $\mu^1$, is not isolated in the sense that any ball around $\mu^1$ contains other Karcher means (taking $p$ close to 0.5), and moreover those Karcher means might be more 'regular' than the barycenter in the sense of having a strictly larger support. Also remark that defining $K:=\{Law((B,r)):r\in I\}$ with $I$ some interval containing $\pm 1$, then $K$ is compact, contains the $\nu_i$'s, and transport maps between elements in $K$ are associative. 
\end{example}

{
The two examples above show that uniqueness of Karcher means (in the sense of Definition \ref{def:Karch_app}), or even uniqueness of these within a nice set $K_\Pi$, is an assumption which can be hardly verified in the general case. Furthermore we encounter the same problem with a localized version of this assumption (for example: ``uniqueness of Karcher means in a small ball intersected with $K_\Pi$''). To complicate the matter further, it is known that in the general case multiple barycentrers may exist. For all these reasons we are inclined to believe that the correct object of study in the general case is the regularized barycenter problem (or more ambitiously, the limit of such regularized problems as the regularization parameter goes to zero at a suitable speed). We refer the reader to the recent works  \cite{Ch23,ChVa23}, and the references therein, for promising results in this direction. }

\bibliography{bibliography}
\bibliographystyle{plain}

\end{document}